\documentclass[10pt]{article}

\usepackage[latin1]{inputenc}
\usepackage{latexsym}
\usepackage{amsmath,amssymb,amsfonts,amsthm}
\usepackage{mathrsfs}
\usepackage{booktabs}
\usepackage{dsfont}
\usepackage{mhequ}
\usepackage[dvipsnames]{xcolor}
\usepackage{verbatim}
\usepackage{tikz}
\usepackage{pgfplots}
\pgfplotsset{compat=1.12}
\usetikzlibrary{patterns}
\usepackage[
]{todonotes}
\usepackage{enumitem}
\usepackage[toc,page]{appendix}

\usepackage[numbers,comma,sort]{natbib}
\usepackage{stmaryrd}
\usepackage{tikz}
\definecolor{db}{RGB}{0, 0, 130}
\usepackage[colorlinks=true,citecolor=db,linkcolor=black,urlcolor=black,pdfstartview=FitH]{hyperref}

\usepackage{pdfsync}
\usepackage{breqn}

\numberwithin{equation}{section}

\newcommand{\R}{\mathbb{R}}
\newcommand{\Rd}{{\R^d}}

\newcommand{\N}{\mathbb{N}}
\newcommand{\EE}{\mathbb{E}}

\def \limbashaut#1#2#3{\mathrel{\mathop{\kern 0pt#1}\limits_{#2}^{#3}}}

\makeatletter
\def\namedlabel#1#2{\begingroup
    #2%
    \def\@currentlabel{#2}%
    \phantomsection\label{#1}\endgroup
}
\makeatother

\makeatletter
\newcommand{\myitem}[1]{%
\item[#1]\protected@edef\@currentlabel{#1}%
}
\makeatother

\newtheorem{definition}{Definition}[section]
\newtheorem{theorem}[definition]{Theorem}
\newtheorem*{theorem*}{Theorem}
\newtheorem{prop}[definition]{Proposition}

\newtheorem{lemma}[definition]{Lemma}

\newtheorem{assumption}[definition]{Assumption}

\author{Konstantinos Dareiotis\footnote{K.Dareiotis@leeds.ac.uk}, El Mehdi Haress\footnote{E.Haress@leeds.ac.uk}, Khoa L\^e\footnote{K.Le@leeds.ac.uk}
\\ University of Leeds, School of Mathematics
} 

\medskip

\date{\empty}

\title{ \Large{\textbf{Uniform pathwise stability of additive singular SDEs driven by fractional Brownian motion}}} 

\begin{document}

\maketitle

\begin{abstract}
We study the long-time behaviour of solutions to a class of $d$-dimensional stochastic differential equations driven by fractional Brownian motion with Hurst parameter $H \in (0,1)$. 
The drift consists of a dissipative Lipschitz term and a singular term of regularity $\gamma >1-1/(2H)$ in Besov-H\"older scales.
We establish 
well-posedness and, through a Markovian enhancement, existence of an invariant measure.
If the singular contribution is sufficiently small, we prove exponential contraction of solutions, and thereby, uniqueness of the invariant measure.
Our methods rely on uniform pathwise estimates which utilise together the dissipativity of the drift and the regularisation effect of the noise.
\end{abstract}


\medskip

\noindent\textbf{MSC2020 subject classification: }
60H10, %SDE
60G22, %fbm
60H50, %reg. by noise
37A25. %ergodicity

\medskip

\noindent\textbf{Acknowledgement: }The authors acknowledge support from the Engineering \& Physical Sciences Research Council (EPSRC), grant number EP/Y016955/1.

\tableofcontents

\section{Introduction}
Let $d\ge1$ be an integer dimension and $t_0\in\R$ be a fixed initial time. 
We consider the following $d$-dimensional  Stochastic Differential Equation (SDE)
\begin{align}\label{eq:sde}
dX_t =  F(X_t) dt +  b(X_t)dt + dB_t, \, t \ge t_0, \  X_{t_0}=x\in\R^d,
\end{align}
where $F$ is a dissipative drift, $b$ is an irregular drift of regularity $\gamma\in(-\infty,1]$ in the Besov-H\"older scales, and $B$ is a $d$-dimensional fractional Brownian motion (fBm) with Hurst parameter $H \in (0,1)$.

Equation \eqref{eq:sde} is motivated in part from the literature on interacting particle systems where the drift is written as the sum of a repulsive force representing the singularity and a conservative force (e.g. Coulomb gases \cite{serfaty2015coulomb}, Dyson Brownian Motion \& Log-Gases \cite{forrester2010log}). Equation \eqref{eq:sde} can also be viewed as a simplification of stochastic partial differential equations of the type 
\begin{align}\label{eq.spde}
\partial u = \Delta u + F(u) + b(u) + \xi 
\end{align}
where $\xi$ is a space-time white noise, $\Delta+F$ plays the role of a dissipative drift, and $b$ is an irregular drift. 
Such models have been studied in  the literature. To name a few, \cite{BounebacheZambotti,athreya2020well} consider the case when  $F=0$ and $b$ is a measure, \cite{nualartpardoux,zambottireflection} consider the case when $b$ is a reflective measure, \cite{scarpa} considers the stochastic Allen--Cahn equation with logarithmic potential.

Our main purpose in the current work is to gain better understanding of the interlay between the different natures of the two drifts. In particular, we focus on well-posedness and the long-time behaviour of solutions to \eqref{eq:sde}. Our main results, which are detailed in Sections \ref{sec:well-posedness} and \ref{sec:longtime}, can be summarised in a simplified form as follows. 
\begin{theorem*}
    \textbf{(I)} If $\gamma >1/2-1/(2H)$ and $F$ has linear growth, then there exists a weak solution to the SDE \eqref{eq:sde}. Moreover, if $\gamma >1-1/(2H)$ and $F$ is Lipschitz continuous, then there exists a unique strong solution to the SDE \eqref{eq:sde}.

    \textbf{(II)} If $\gamma >1-1/(2H)$ and $F$ is Lipschitz continuous, then there exists an invariant measure associated to \eqref{eq:sde}. Moreover, if the Besov--H\"older norm of $b$ is sufficiently small, then moments of the difference of any two solutions decay exponentially and consequently, the invariant measure is unique.
\end{theorem*}
% In Section \ref{sec:well-posedness} and Section \ref{sec:longtime}, we give detailed versions of the above Theorem and their proofs.

In the absence of one of the drifts, equation \eqref{eq:sde} has been studied extensively in the literature. 
When $F=0$, due to the regularisation effect from the noise, equation \eqref{eq:sde} is well-posed even for non-Lipschitz drift $b$ (i.e. $\gamma<1$). 
In fact, first results for Brownian SDEs can be traced back to Zvonkin \cite{Zvonkin} and Veretennikov \cite{Veretennikov} who considered well-posedness of such SDEs with bounded measurable drifts. Their results were subsequently extended by Krylov and R\"ockner \cite{KrylovRockner} for SDEs with integrable drifts under the so-called subcritical Ladyzhenskaya-Prodi-Serrin condition. 
Weak well-posedness for Brownian SDEs with distributional drifts have been considered in \cite{portenko1990generalized,FRW,BassChen,DelarueDiel,FIR,LeGall,Lejay}. 
SDEs with irregular drifts driven by fractional Brownian motion also received considerable attention, with early contributions by Nualart and Ouknine \cite{NualartOuknine}.  The emergence of new methods such as sewing techniques have ignited renewed interests on such equations, resulting in a host of new results. Starting from  Catellier and Gubinelli in \cite{CatellierGubinelli}, existence and uniqueness of strong solutions for a new class of equations have been established. We refer to \cite{GaleatiGerencser} and the references therein for the most recent advancements. The emerging techniques also allow to identify the probability distribution of the solutions (i.e. weak well-posedness), see \cite{ButkovskyLeMytnik}.

When $b=0$ and $B$ is a Brownian motion, well-posedness and long time behaviour of equation \eqref{eq:sde} are well-understood, see \cite{markovchains} for an exhaustive analysis. For such equations, it is typically expected that the invariant measure exists uniquely and that the convergence of the law of solutions towards the stationary state is exponential (in the total variation norm). When $B$ is a fractional Brownian motion, due to long range correlations, solutions to \eqref{eq:sde} are not  Markov processes. To study the long time behaviour of this class of non-Markovian processes, a prominent general framework called stochastic dynamical system has been introduced by Hairer \cite{Hairer}. Therein, the author established the existence and uniqueness of invariant measures, and the convergence rate of the law of an arbitrary solution towards the stationary state. This framework has been developed further in \cite{hairerohashi,hairer9,hairerpillai} and used for proving estimates on the density of the stationary state (see \cite{li2023non} and the references therein) and for approximating the invariant measure (see \cite{cohen,panloup2020general} and references therein).

Few works consider a general framework that combines both a dissipative drift and a singular drift, with the exception of \cite{bao2024limit} and \cite{haress2024numerical}. In the Ph.D. thesis of the second author (see \cite[Chapter 6]{haress2024numerical}), a uniform-in-time bound on the moments of the solution was established for the stochastic heat equation with a distributional drift and a Lipschitz dissipative drift. This bound was for $L_q$-norms in space of the solutions which was not enough to deduce existence of an invariant measure in that case. In \cite{bao2024limit},  the authors consider the case when $B$ is a Brownian motion, $F$ is locally Lipschitz and weakly dissipative, and $b$ is H\"older continuous. They establish existence and uniqueness of an invariant measure and an exponential contraction in the $1$-Wasserstein distance which helps them establish limit theorems. 
Compared to \cite{bao2024limit}, we work in a general setting where $B$ is a fractional Brownian motion and $b$ can be a distribution (for small $H$). We require stronger assumptions in exchange for a stronger result (decay in moments).

To analyse the invariant measures, it is tempting to adopt the framework of stochastic dynamical systems from \cite{Hairer}. However, putting \eqref{eq:sde} into such general framework demands, in particular, that the solution is a continuous function of the driving noise, which is a stringent property in the presence of the irregular drift.
We therefore adopt a slightly different perspective from previous literatures, relying on an observation from \cite{Hairer} that while \eqref{eq:sde} is non-Markovian, it is still possible to couple the solution with the historical past of the driving noise so that the resulted enhanced process is Markovian. In this way, concepts of invariant measures for Markov processes can be naturally carried forward. 
We find that this point of view offers a direct treatment to the long-time behaviour of \eqref{eq:sde}, without the need for an abstract framework, and thus provides some simplifications over existing literature.

Upon completion of this paper, we are aware of a parallel development \cite{Avi} where the authors also study the long-time behaviour of \eqref{eq:sde}.

\paragraph{Overview of proofs.}
To simplify our discussion and focus only on the key ideas, we will assume that $b$ is a continuous bounded function. To construct weak solutions (part \textbf{(I)}), we consider an approximating equations with smooth drifts and establish, using linear growth condition and stochastic sewing lemma (\cite{le2020stochastic}), Davie-type estimates (\cite{Davie}) which depends only on the corresponding Besov-H\"older norm of $b$. A tightness argument then implies the existence of a weak solution. 
In the case when $F$ is Lipschitz, the stochastic sewing techniques also yield a moment comparison between two solutions, which implies strong uniqueness. 

Concerning part \textbf{(II)}, first, we put \eqref{eq:sde} into the standard framework of Markov processes by considering an enhanced dynamic consisting of  the solution and the historical past of the driving noise. We show that this enhanced dynamic is a Markov process. An invariant measure associated to \eqref{eq:sde} is defined as the invariant measure of the enhanced process starting from a generalised initial condition (a concept we adopted from \cite{Hairer}). 
To establish existence and uniqueness of invariant measures for the enhanced process, we employ  standard approaches which rely on uniform moment estimates.
To this end, we   assume that there exists $\kappa_1 >0$ such that 
\begin{align}\label{eq:strong-dissipativ}
\langle F(x) - F(y) , x-y \rangle \leq - \kappa_1 | x-y |^2  \quad \forall x,y \in \R^d .
\end{align}
To show that the moments of the solution to \eqref{eq:sde} are uniformly bounded in time (thereby showing existence of an invariant measure), we compare \eqref{eq:sde} with the Ornstein-Uhlenbeck (OU) process
\begin{align}\label{eq:U}
d U_t = -U_t dt + dB_t , \quad U_{t_0}=x.
\end{align}
Using \eqref{eq:strong-dissipativ} and linear growth condition, we have that
\begin{align*}
|X_t-U_t|^2 & \leq C \Big( 1+ \int_{t_0}^t e^{-\kappa_1(t-r)} |U_r|^2 dr + \int_{t_0}^t e^{-\kappa_1(t-r)} \langle X_r-U_r, b\left(X_r\right)\rangle dr \Big) .
\end{align*}
Here if $b$ is bounded, it is straightforward to conclude from the above inequality that the moments of $X-U$, and hence moments of $X$, are bounded uniformly in time. However, these estimates can not be extended when $b$ is a distribution. 
To proceed, we utilise the regularisation effect of the noise through stochastic sewing techniques. We show (in Proposition \ref{prop:reg3}) that any (positive) moment of $\int_s^t e^{\kappa_1(t-r)} \langle X_r-U_r, b(X_r) \rangle$ is majored by  moments of $X-U$ with a multiplicative constant $C$ that depends only on the corresponding Besov-H\"older norm of $b$. From here, a Gr\"onwall-type argument shows that the moments of $X-U$, and hence moments of $X$, are bounded uniformly in time.

Similar arguments are applicable to show decaying moment estimates for the difference of two solutions $X,Y$ to \eqref{eq:sde}. Indeed, we have analogously to previous argument that 
\begin{align*}
| X_t-Y_t|^2 \leq e^{-2\kappa_1 (t-t_0)}|X_{t_0}-Y_{t_0}|^2 + \int_{t_0}^t e^{-2\kappa_1 (t-r)}\langle X_r-Y_r, b(X_r)-b(Y_r)  \rangle dr .
\end{align*}
In the case when $b$ is Lipschitz with Lipschitz constant $L_b$, it is straightforward to obtain from the above inequality that 
\begin{align*}
    |X_t-Y_t|^2\leq e^{-2(\kappa_1-L_b)(t-t_0)}|X_{t_0}-Y_{t_0}|^2.
\end{align*}
This shows that if $L_b<\kappa_1$, then the moments of $X_t-Y_t$ decay exponentially. To obtain similar result when $b$ is in the Besov-H\"older scales, we utilise the regularisation effect of the noise once again. The idea is that stochastic sewing techniques allow us to formally say that the map $x \mapsto \int b(B+x)$ is Lipschitz in $x$  with Lipschitz constant $\tilde L$ proportional to the corresponding Besov-H\"older norm of $b$. Hence, we can perform an additional Gr\"onwall-type argument to conclude that the moments of $| X_t-Y_t|^2$ grows exponentially with rate $-c(\kappa_1-\tilde L)$  
for some constant $c>0$. If  $\tilde{L}$ is small enough, this implies exponential decay, and thereby showing that \eqref{eq:sde} has a unique invariant measure. 

Note that the positivity of the constant $\kappa_1$ in \eqref{eq:strong-dissipativ} only matters for long-time analysis. Hence, in some intermediary results, e.g. existence of weak solutions for instance, we merely assume $\kappa_1$ to be a real number. Moreover, the usual stability condition on $F$ involves an additional constant $\kappa_2 \ge 0$ in the right-hand side of \eqref{eq:strong-dissipativ} (e.g. \cite{Hairer}), hence some intermediary results (existence of weak solutions and the uniform bound on the moments) will also involve this constant.

The arguments presented above are of course over-simplified. Additional cares are taken into account when $b$ is merely a distribution. This involves working with a proper concept of solutions to \eqref{eq:sde}  and controlling H\"older norms resulted from applying sewing techniques. 

\paragraph{Applications.}

Our results and proofs can be useful in two main directions.

\emph{Long-time numerical approximation.}
Numerical schemes for \eqref{eq:sde} with $F=0$ have been studied for example by \cite{haress2024numerical} in finite time, but understanding their long-time behaviour in the presence of $F$ and for $\gamma <0$ remains largely open. 
Our analysis of the invariant measure and exponential moment decay provides a natural theoretical foundation for studying the ergodic properties of discrete-time approximations. 
In particular, if $(b^n)_{n \in \N}$ is a smooth sequence that converges to $b$ in the Besov-H\"older space of regularity $\gamma$, one might consider (as in \cite{haress2024numerical}) the following tamed scheme for $h \in (0,1)$,
\begin{align}\label{eq:scheme}
    X^{n,h}_t=x+\int_{t_0}^t F(X^{n,h}_{r_h})dr+\int_{t_0}^t b(X^{n,h}_{r_h})dr+B_t-B_{t_0}, \, t \ge t_0,
\end{align}
where $r_h=h \lfloor r/h \rfloor$. Analysing the long-time properties of $X^{n,h}$ would then require proving new regularisation properties of the discretised fBm that are similar to the ones presented in this paper for the fBm. The goal here is to obtain a uniform-in-time bound on the error between $X^{n,h}$ and $X$. This would also lead to an approximation of the invariant measure, and extend to the singular regime the long-time approximation results known for smooth $b$.

\emph{Parameter estimation from invariant measures.}
A second natural application concerns statistical inference.
Since the invariant measure characterises the long-time law of the process, it can be used to estimate a model parameter $\theta_0$, for instance a coefficient appearing in the drift term $(F+b)$.
Given long-time discrete observations $(X_{t_k})_{k=0,\cdots,n}$, one can define an empirical random measure $\mu_N := \frac{1}{N}\sum_0^N \delta_{X_{t_k}}$ and infer the unknown parameter $\theta_0$ by minimising a suitable distance between $\mu_N$ and the theoretical invariant measure $\mu_\theta$ (where $\theta$ denotes a general parameter). Moreover, in cases where the invariant measure is unknown, $\mu_\theta$ can be approximated using the scheme \eqref{eq:scheme} enabling a fully implementable estimation procedure. We refer to \cite{panloup2020general} and references therein for the case $b=0$ but the general case still needs to be developed and would require ergodic results which our paper provides a rigorous basis for.

\paragraph{Organisation of the paper.} In Section \ref{sec:notation}, we introduce notation used throughout the paper, give the definition of a solution and state the main assumptions on $F$ and $b$.

In Section \ref{sec:well-posedness}, we study the well-posedness of the SDE \eqref{eq:sde}. We start by stating the main results, and in Section \ref{sec:regprop} we recall and prove regularisation properties of the fractional Brownian motion (fBm). In Section \ref{sec:apriori}, we establish a priori estimates for solutions. First, we derive regularity estimates for solutions with smooth drift $b$, then we prove general stability results (with respect to the initial condition and the drift) for solutions that satisfy some given regularity. In Section \ref{sec:weakexistence}, we prove existence of weak solutions. In Section \ref{sec:strongexistence}, we prove existence of solutions and uniqueness in a class of H\"older continuous processes and describe how the difference of two solutions started from different initial conditions evolves in moments.

In Section \ref{sec:longtime}, we study the existence and uniqueness of invariant measures associated to \eqref{eq:sde}. We consider an SDE where the history of the fBm is fixed to be equal to an $\mathcal{F}_{t_0}$ deterministic path and use it to define a Markov evolution which consists of coupling the solution with the evolution of the history of the fBm. We state the main results which include well-posedness of the aforementioned SDE, a proof of the Markov property, and existence and uniqueness of the invariant measure. The last two results are proved respectively in Section \ref{subsec:link-SDEs} and Section \ref{sec:inv}.

In Appendix \ref{app:A}, we gather a useful lemma on the regularisation properties of the fBm, a version of the stochastic sewing lemma, and a disintegration result.

\section{Notation, definitions and assumptions}\label{sec:notation}
In this section, we introduce some notation used throughout the paper, define a solution to the equation \eqref{eq:sde} and give the assumptions on $F$ and $b$.

Let $(\Omega,\mathcal{F},\mathbb{F} = (\mathcal{F}_{t})_{t\in \R},\mathbb{P})$ be a filtered probability space which satisfies the usual conditions. The conditional expectation given $\mathcal{F}_{t}$ is denoted by $\mathbb{E}^{t}$ when there is no risk of confusion on the underlying filtration. The law of a random variable $X$ is denoted by $\mathcal{L}(X)$. For any $m \in [1,\infty]$, the $L_m(\Omega)$ norm of $X$ is denoted by $\|X\|_{L_m}$ and the space $L_m(\Omega)$ is simply denoted by $L_m$.

A  process $(W_t)_{t \in \R}$ is a standard $\mathbb{F}$-Wiener process if $W_0=0$, $W$ has almost surely continuous paths, $W$ is adapted to $\mathbb{F}$ and for all $s < t$, the increment $W_t-W_s$ is independent of $\mathcal{F}_s$ and $W_t-W_s \sim \mathcal{N}(0,(t-s)I_d)$.

We use the Mandelbrot and Van Ness representation of the fractional Brownian motion as an integral with respect to an $\mathbb{F}$-Wiener process $W$, that is
\begin{align}\label{eq:MVN}
    B_{t} = \alpha_H \int_{-\infty}^{0} \Big((t-u)^{H-\frac{1}{2}}-(-u)^{H-\frac{1}{2}}\Big) d W_u+\alpha_H \int_{0}^{t}(t-u)^{H-\frac{1}{2}} d W_u . 
\end{align}
In particular, there is a normalisation constant $\alpha_H$ such that for any $t_0 \in \R$ and $t \ge t_0$,
\begin{align}
B_{t}-B_{t_0} & =\alpha_H \int_{-\infty}^{t_0} \Big((t-u)^{H-\frac{1}{2}}-(t_0-u)^{H-\frac{1}{2}}\Big) d W_u+\alpha_H \int_{t_0}^{t}(t-u)^{H-\frac{1}{2}} d W_u \nonumber \\
& =: \bar{B}_{t}^{t_0}+\widetilde{B}_{t}^{t_0}. \label{eq:decomp-fBm}
\end{align}
The processes 
$\bar{B}^{t_0},\widetilde{B}^{t_0}$ are respectively called the history and the innovation at time $t_0$. 

Let $f\colon\R^d\to \R^d$ be a Borel-measurable function. We denote the supremum norm of $f$ by 
$\|f \|_{\infty} = \sup_{x \in \R^d} |f(x)|$.
For $\alpha \in (0,1]$, we denote by $\mathcal{C}^\alpha$ the space of bounded $\alpha$-H\"older continuous functions and denote the $\mathcal{C}^{\alpha}$ norm of $f$ by $\|f\|_{\mathcal{C}^\alpha} = \| f \|_\infty + \sup_{x \neq y} \frac{|f(x)-f(y)|}{|x-y|^\alpha}$. Moreover, for any interval $I$, we denote by $\mathcal{C}^\alpha_{I}$ the space of H\"older continuous functions $f: I \rightarrow \R^d$. For $\alpha < 0$, we say that a Schwartz-distribution $f$ is of class $\mathcal{C}^{\alpha}$ if
\begin{align}\label{eq:alpha-norm}
\| f \|_{\mathcal{C}^\alpha} := \sup_{\varepsilon \in (0,1]} \varepsilon^{-\alpha/2} \| P_\varepsilon f \|_{{\infty}} < +\infty ,
\end{align}
where $P_\varepsilon$ is the Gaussian semigroup defined by $P_{\varepsilon} f(x) := \int_{\R^d} p_{\varepsilon}(x-y) f(y) dy$ and $p_{\varepsilon}(x)=\frac{1}{(4\pi \varepsilon )^{d/2}} \exp\big(-\frac{|x|^2}{4\varepsilon} \big)$. We write $\mathcal{C}^0$ for the space of bounded measurable functions. 
We also write $\mathcal{C}^\infty$ for the space of bounded smooth functions whose  derivatives (of all orders) are also bounded. Finally, for any $\alpha \in (-\infty,1]$, we denote by $\mathcal{C}^{\alpha+}$ the closure of $\mathcal{C}^\infty$ in $\mathcal C^{\alpha}$.  

For an interval $I$ of $\R$,  we denote by $\mathcal{C}(I,\R^d)$ the space of continuous functions $f:I\rightarrow \R^d$ equipped with the metric
\begin{align}\label{eq:metric}
    d(f,g) = \sum_{k} 2^{-k} \Big( 1 \wedge \sup_{r \in [-2^k,2^k] \cap I} | f(r)-g(r)| \Big),
\end{align}
and define the simplex $\Delta_{I} = \{ (s,t) \in I^2: s < t \}$. Moreover, for any $m \in [2,\infty)$ and $q \in [2,\infty]$,
we introduce the following notation for any random variable $X$ 
\begin{align}\label{def:cond-norm}
\| X \|_{L_{m,q}^{\mathcal{F}_s}} = \| \left( \EE^s |X|^m \right)^{1/m} \|_{L_q} .
\end{align}
When $m=q$, such quantity deduces to the usual $L_m$ norm.
Let $\phi:I \times \Omega \rightarrow \mathbb{R}$, we introduce the following semi-norms:
\begin{align}
[\phi]_{\mathcal{C}_{I}^{\alpha} L_{m,q}}=\sup _{\substack{(s, r) \in \Delta_{I}}} \frac{\|\phi_r-\phi_s\|_{L_{m,q}^{\mathcal{F}_s}} }{|r-s|^\alpha}, \label{eq:semi-norm} \\
\llbracket \phi \rrbracket_{\mathcal{C}_{I}^{\alpha} L_{m,q}}=\sup _{\substack{(s, r) \in \Delta_{I}}} \frac{\|\phi_r-\EE^s \phi_r\|_{L_{m,q}^{\mathcal{F}_s}} }{|r-s|^\alpha} \label{eq:semi-norm-2}.
\end{align}
Here we have omitted the dependence on the probability space $(\Omega,\mathcal{F},\mathbb{F},\mathbb{P)}$ in the notation, which will be clear from the context. 
When $m=q$, we simply write $[\phi]_{\mathcal{C}_{I}^{\alpha} L_{m}} := [\phi]_{\mathcal{C}_{I}^{\alpha} L_{m,m}}$ and $\llbracket \phi \rrbracket_{\mathcal{C}_{I}^{\alpha} L_{m}} := \llbracket \phi \rrbracket_{\mathcal{C}_{I}^{\alpha} L_{m,m}}$. 

Notice that by the triangle inequality one always has 
$$(\EE^s|\phi_r-\EE^s\phi_r|^m )^{1/m} \leq (\EE^s|\phi_r-\phi_s|^m)^{1/m} + (\EE^s|\EE^s(\phi_s-\phi_r)|^m)^{1/m},$$
which implies that
\begin{align}\label{eq:compar-norms}
\llbracket \phi \rrbracket_{\mathcal{C}_{I}^{\alpha} L_{m,q}} \leq 2[\phi]_{\mathcal{C}_{I}^{\alpha} L_{m,q}}.
\end{align} 

Let $(E,d)$ be a metric Polish space. We write $\mathcal{M}_1(E)$ for the set of probability measures on $E$. 
For $m \in (1,\infty)$, we will consider the $m$-Wasserstein distance, which is defined for every $\mu,\nu$ in $\mathcal{M}_1(E)$ as follows
\begin{align}\label{eq:wasserstein}
\mathcal{W}_m(\mu,\nu) = \inf\{ (\EE d(X,Y)^m)^{1/m}\},
\end{align}
where the infimum is taken over all random variables $X,Y$ defined on a common probability space such that $(X,Y)$ is a coupling of $\mu$ and $\nu$.

~

Finally, we denote by $C$ a constant that can change from line to line and that does not depend on any parameter other than those specified in the associated lemma, proposition or theorem. When we want to make the dependence of $C$ on some parameter $a$ explicit, we will write $C(a)$.

~

We give a meaning to equation \eqref{eq:sde} with distributional drift by approximating the drift by a smooth sequence.

\begin{definition}\label{defsol-SDE}
Let $(\Omega,\mathcal{F},\mathbb{F},\mathbb{P})$ be a filtered probability space, and let $W$ be an $\mathbb{F}$-Wiener process.
Let $x \in \R^d$, $t_0 \in \R$, $T > t_0$, $\gamma \in \mathbb{R}$, $b \in \mathcal{C}^{\gamma+}$. We say that $X$ is a solution to \eqref{eq:sde} with respect to $(\Omega,\mathcal{F},\mathbb{F},\mathbb{P})$ and $W$ on $[t_0,T]$ if $X$ is adapted to the filtration $\mathbb{F}$ and
\begin{enumerate}
\item there exists an $\R^d$-valued $\mathbb{F}$-adapted process $(K_t)_{t \in [t_0,T]}$ such that, a.s., for all $t \in [t_0,T]$
\begin{equation}\label{solution1}
X_t=x+\int_{t_0}^t F(X_r) dr + K_t+B_t-B_{t_0}  \text{,}
\end{equation}
where $B$ is a fractional Brownian motion given by $W$ via the formula \eqref{eq:MVN}\text{;}

\item for every sequence $(b^k)_{k\in \mathbb{N}}$ of smooth bounded functions converging to $b$ in $\mathcal{C}^{\gamma}$, we have
\begin{equation}\label{approximation2}
       \sup_{t\in [t_0,T]}\left|\int_{t_0}^t b^k(X_r) dr 
       -K_t\right|   \underset{k\rightarrow \infty}{\longrightarrow} 0 ~\text{ in probability}.
\end{equation}
\end{enumerate}
\end{definition}

A typical example of such sequence is $(b^k)=(P_{1/k}b)$. In the paper, we work under the Catellier-Gubinelli condition $\gamma> 1-1/(2H)$, so as in the definition above, we will assume without any loss of generality that $b \in \mathcal{C}^{\gamma+}$. In fact, if $b$ is in $\mathcal{C}^\gamma$ then, for any $\varepsilon>0$, $b$ is in $\mathcal{C}^{(\gamma-\varepsilon)+}$. Therefore, we make the following assumption on $b$:
\begin{assumption}\label{assumpb}
$b$ belongs to the space $\mathcal C^{\gamma+}$ and there exists a positive constant $\Xi$ such that
\begin{align*}
    \| b \|_{\mathcal{C}^\gamma} \leq \Xi .
\end{align*}
\end{assumption}

The general stability and growth assumptions on the dissipative drift $F$ will be the following:

\begin{assumption}\label{assumpF}
$F: \R^d \rightarrow \R^d$ is a continuous function and the following hold:
\begin{enumerate}[label=(\roman*)]
\item There exists $\kappa_1 \in \R$ and $\kappa_2 \ge 0$, such that
\begin{align}\label{eq:Fdissipative}
\langle F(x) - F(y) , x-y \rangle \leq - \kappa_1 | x-y |^2 + \kappa_2  \quad \forall x,y \in \R^d .
\end{align}
\item There exists a constant $\kappa_3>0$, such that $F$ satisfies
\begin{align}\label{eq:growth}
| F(x) | \leq \kappa_3(1+|x|) .
\end{align}
\end{enumerate}
We write $\kappa := (\kappa_1,\kappa_2,\kappa_3)$.
\end{assumption}

\begin{assumption}\label{Flip}
$F: \R^d \rightarrow \R^d$ is a Lipschitz continuous function with Lipschitz constant $L_F$.  
\end{assumption}
Notice that Assumption \ref{Flip} implies Assumption \ref{assumpF} with $\kappa_1=-L_F$, $\kappa_2=0$ and $\kappa_3=|F(0)|+L_F$. We will make  use of both assumptions in the paper to distinguish between regularity assumptions and stability assumptions. In particular, for the long-time analysis, we will assume $\kappa_1>0$, in which case \eqref{eq:Fdissipative} does not follow from Assumption \ref{Flip}.

\section{Well-posedness and regularity of the solution}\label{sec:well-posedness}
In this section, we deal with the well-posedness of the SDE \eqref{eq:sde} and the regularity of the solutions. 

\begin{theorem}\label{thmmain-existence} 
Let $\gamma \in (1/2-1/(2H),1)$. Assume that $F,b$ satisfy Assumptions \ref{assumpF}, \ref{assumpb}, respectively.
Let $x \in \R^d$, $T \in \R$ and $t_0 < T$. Then there exists a filtered probability space $(\Omega,\mathcal{F},\mathbb{F},\mathbb{P})$ and an $\mathbb{F}$-Wiener process $W$ such that there exists a solution $(X_t)_{t \in [t_0,T]}$ to \eqref{eq:sde} with respect to $(\Omega,\mathcal{F},\mathbb{F},\mathbb{P})$ and $W$. \\
 
Moreover, for any $m \in [2,\infty)$, there exist constants $C_1:=C_1(m,\gamma,H,x,\kappa,\Xi)$ and $C_2:=C_2(m,\gamma,H,\kappa,\Xi)$ such that for any $t_0 \leq s <t$ with $t-s \leq 1$, $X$ satisfies 
\begin{align}
 \sup_{r \in [t_0,T]} \| X_r  \|_{L_m} & \leq C_1 \Big(1+ (T-t_0) \wedge \frac{1-e^{-\kappa_1(T-t_0)} }{\kappa_1} \Big)  \label{eq:thm-existence-1},\\
  [X-B]_{\mathcal{C}^{1+H(\gamma \wedge 0)}_{[s,t]} L_m} & \leq C_2 \big( 1+\sup_{r \in [s,t]} \|X_r\|_{L_m} \big)\label{eq:thm-existence-2} .
 \end{align}
Finally, if $F$ satisfies Assumption \ref{Flip}, then there exists a constant $C:=C(m,\gamma,H, L_F)$ such that for any $t_0 \leq s <t$ with $t-s \leq 1$, $X$ also satisfies
\begin{align}\label{eq:thm-existence-3}
\llbracket X-B \rrbracket_{\mathcal{C}^{1+H(\gamma \wedge 0)}_{[s,t]} L_{m,\infty}} \leq C .
\end{align}
 \end{theorem}
 If $\gamma \in [0,1)$, then $b$ is a bounded function and the above results are known with $F$ having a polynomial growth instead of a linear one. In fact one can check that \eqref{eq:thm-existence-2} and \eqref{eq:thm-existence-3} always hold and \eqref{eq:thm-existence-1} follows from a comparison with the OU process (following the arguments presented in the paragraph ``Overview of some key estimates" in the introduction). However, we still include the case $\gamma \in [0,1)$ in Theorem \ref{thmmain-existence} as it will be relevant in the next results.

If $\kappa_1>0$, then \eqref{eq:thm-existence-1} becomes a uniform-in-time bound on the moments which will be useful in proving existence of an invariant measure. Moreover, if $\kappa_1 >0$, then combining \eqref{eq:thm-existence-1} and \eqref{eq:thm-existence-2}, one obtains a uniform H\"older regularity over small intervals.

The proof of Theorem \ref{thmmain-existence} is established in Section \ref{sec:weakexistence}. In the second theorem, we state existence and uniqueness of solutions and we describe how the difference of any two solutions behaves in the long-time.

\begin{theorem}\label{thmmain-uniqueness}
Let $(\Omega,\mathcal{F},\mathbb{F},\mathbb{P})$ be a filtered probability space and let $W$ be an $\mathbb{F}$-Wiener process. Let $B$ be a fractional Brownian motion given by \eqref{eq:MVN}.
Let $\gamma \in (1-1/(2H),1)$. Assume that $F$ satisfies Assumption \ref{Flip} and $b$ satisfies Assumption \ref{assumpb}.
Let $x \in \R^d$, $t_0 \in \R$ and $T > t_0$, then there exists a unique solution $(X_t)_{t \in [t_0,T]}$ to \eqref{eq:sde}, with respect to $(\Omega,\mathcal{F},\mathbb{F},\mathbb{P})$ and $W$, in the class: 
 \begin{align}\label{eq:class-uniqueness}
 \mathcal{V} = \{ Y: \,  [ Y-B ]_{\mathcal{C}^{1/2}_{[t_0,T]} L_2} < \infty \}.
 \end{align}
Moreover, $X$ satisfies \eqref{eq:thm-existence-2} and \eqref{eq:thm-existence-3}. Furthermore, let assume that $F$ satisfies Assumption \ref{assumpF}, then $X$ also satisfies \eqref{eq:thm-existence-1}, and if $\kappa_2=0$, then for any $m \in [2,\infty)$, there exists a universal constant $C$ and a constant $\mathbf{M} := \mathbf{M}(m,\gamma,H, L_F,\Xi)$ such that for 
\begin{align}\label{eq:defbeta}
\beta := \beta(m,\|b \|_{\mathcal{C}^{\gamma}} ) := -\kappa_1+\mathbf{M}(m,\gamma,H, L_F,\Xi)\|b \|_{\mathcal{C}^{\gamma}} ,
\end{align}
and any solution $Y$ in the class $\mathcal{V}$ with initial condition $y$, we have
\begin{align}\label{eq:exp-decay}
\| X_t-Y_t\|_{L_m}\leq C |x-y| e^{\beta (t-t_0)} ,\, \forall t \ge t_0 .
\end{align}
\end{theorem} 
In the particular case where $\mathbb{F}$ is the filtration generated by $W$, one obtains a unique solution that is adapted to $\mathbb{F}$.
The class of uniqueness defined in Theorem \ref{thmmain-uniqueness} is the same class where uniqueness is stated when $F=0$, see for example \citep{GHR2024}.
The proof of Theorem \ref{thmmain-uniqueness} is established in Section \ref{sec:strongexistence}. The universal constant $C$ and the constant $\mathbf{M}$ can be tracked to the proof of Proposition \ref{prop:stab-lip} and can be explicitly computed.

Note that the well-posedness of the SDE when $F=0$ is also known in the limit case $\gamma=1-1/(2H)$ if $b$ is a distribution in the non-homogeneous Besov space $\mathcal{B}^{\tilde{\gamma}}_{p,\infty}$ with $\tilde{\gamma}-d/p = \gamma$ and $p<\infty$. In this case, the regularisation properties are based on a critical version of the stochastic sewing lemma that involves a logarithmic term in the right-hand side. While one could establish similar results here in finite time, we are more interested in the long-time behaviour and the exponential convergence towards the invariant measure. More specifically, this logarithmic dependence would appear in the proof of Proposition \ref{prop:stab-lip} (see \eqref{eq:X-Y-progress}) and would lead to a growth of $g(t)=e^{-\kappa_1 t} \| X_t-Y_t \|$ that is like a double exponential, which would not allow us to conclude that the difference of two solutions goes to zero even if the Besov-H\"older norm of $b$ is small.

In order to prove the main results, we start by proving regularisation properties of the fBm in Section \ref{sec:regprop} and a priori estimates in Section \ref{sec:apriori}.

\subsection{Regularisation properties}\label{sec:regprop}
Let $(\Omega, \mathcal{F}, \mathbb{F}, \mathbb{P})$ be a probability space, $W$ be an $\mathbb{F}$-Wiener process and $B$ be the fractional Brownian motion defined by \eqref{eq:MVN}. Throughout the section, we fix $t_0 \in \R$ and write for simplicity $\widetilde{B}_t:= \widetilde{B}^{t_0}_t$ for $t\ge t_0$.
In this section, we establish  quantitative estimates on the regularisation effects of the fractional Brownian motion, which are used later to obtain a priori estimates and to prove the main results. 
To be more precise, given a smooth function $f$, a constant $c\ge0$, and $\mathbb{F}$-adapted processes $\phi,\psi$, we obtain estimates for moments of the integral $\int e^{cr} \langle \phi_r, f(\psi_r+\widetilde{B}_r) \rangle dr$ and study its stability with respect to $\phi,\psi$. The regularisation effect is reflected through the fact that these estimates continue to hold with much more general regularity condition on $f$. 

The first two propositions recall regularisation properties (now stated for the process $\widetilde{B}$) of the map $x \mapsto \int f(\widetilde{B}_r+x) dr$. 

\begin{prop}\label{prop:reg1}
Let $\gamma \in (-1 /(2 H), 1)$, $m \in [2,\infty)$ and $q \in [m,\infty]$. Let $\tau \in (0,1)$ be such that $H(\gamma-1)+\tau >0$. There exists a constant $C:=C(q,m,\tau,\gamma,H)>0$ such that for any $f \in \mathcal{C}^{\infty}$, any $\mathbb{F}$-adapted process $\psi: [t_0,T] \times \Omega \rightarrow \mathbb{R}^d$, any $t_0 \leq s < t$, we have
\begin{align}\label{eq:reg1}
\left\|  \int_s^t f(\psi_r+\widetilde{B}_r) dr  \right\|_{L_{m,q}^{\mathcal{F}_s}} \leq C \| f \|_{\mathcal{C}^\gamma} (t-s)^{1+H(\gamma \wedge 0)} + C \| f \|_{\mathcal{C}^\gamma} \llbracket \psi \rrbracket_{\mathcal{C}^{\tau}_{[s,t]} L_{m,q}} (t-s)^{1+H(\gamma-1)+\tau}.
\end{align}
\end{prop}
\begin{proof}
For $\gamma <0$, in \cite[Proposition 4.5$(a)$]{GHR2024}, the same estimate is given for $\int f(\psi_r+B_r)dr$ instead of $\int f(\psi_r+\widetilde B_r)dr$. Hence, one can apply their result for the process $\psi-\bar{B}^{t_0}$, noting that $(\psi-\bar{B}^{t_0})+B=\psi-\widetilde{B}$ and $\llbracket \psi-\bar{B}^{t_0} \rrbracket_{\mathcal{C}^{\tau}_{[s,t]} L_{m,q}}= \llbracket \psi \rrbracket_{\mathcal{C}^{\tau}_{[s,t]} L_{m,q}}$ to obtain \eqref{eq:reg1}.
For $\gamma \ge 0$, the result follows trivially by using the supremum norm of $f$.
\end{proof}

The following proposition gives a quantitative analysis of the bound in Proposition \ref{prop:reg1} and is borrowed from \cite{GHR2024}.

\begin{prop}\label{prop:reg2}
Let $\gamma \in (1-1 /(2 H), 1)$ and $m \in [2,\infty)$. Let $\tau \in (0,1)$ such that $(\tau \wedge 1/2) + H(\gamma-1)>0$. There exists a constant $C:=C(m,\gamma,H,\tau)>0$ such that for any $f \in \mathcal{C}^{\infty}$, any $\mathbb{F}$-adapted processes $\psi, \phi: [t_0,T] \times \Omega \rightarrow \mathbb{R}^d$, any $t_0 \leq s < t$ with $t-s \leq 1$, we have
\begin{align}\label{eq:reg2}
\begin{split}
& \left\|  \int_s^t f(\psi_r+\widetilde{B}_r)-f(\phi_r+\widetilde{B}_r) dr  \right\|_{L_m} \\ & \leq C \| f \|_{\mathcal{C}^\gamma}   \llbracket \psi-\phi \rrbracket_{\mathcal{C}^{\tau}_{[s,t]}L_m}  (t-s)^{1+\tau+H(\gamma-1)} \\
& + C \|f \|_{\mathcal{C}^\gamma} (1+ \llbracket \psi \rrbracket_{\mathcal{C}^{1+H (\gamma \wedge 0)}_{[s,t]}L_{1,\infty}} ) \sup_{r \in [s,t]}\|\psi_r-\phi_r\|_{L_m} (t-s)^{1+H(\gamma-1)} .
\end{split}
\end{align}
\end{prop}

\begin{proof}
In \cite[Proposition 4.7]{GHR2024}, a similar estimate is given for $\int f(\psi_r+B_r)-f(\phi_r+B_r)dr$. Here, the proof follows the same lines but applied for the processes $\psi-\bar{B}^{t_0}$ and $\phi-\bar{B}^{t_0}$. The differences are: we keep track of the two powers in $(t-s)$ when applying the stochastic sewing lemma and we state the result with the semi-norm \eqref{eq:semi-norm-2} instead of the classical H\"older norm. Noting that $(\cdot-\bar{B}^{t_0})+B=\cdot-\widetilde{B}$ and $\llbracket \cdot-\bar{B}^{t_0} \rrbracket_{\mathcal{C}^{\tau}_{[s,t]} L_{m,q}}= \llbracket \cdot \rrbracket_{\mathcal{C}^{\tau}_{[s,t]} L_{m,q}}$, one can check that the proof goes through.
\end{proof}

The next two propositions are useful to get long-time properties of the solutions to \eqref{eq:sde}. 

\begin{prop}\label{prop:reg4}
Let $\gamma  \in(-1/(2H),1)$ and $m \in[2, \infty)$. There exists a constant $C:=C(m,\gamma,H)>0$ such that for any $c \ge 0$, $f \in$ $\mathcal{C}^{\infty}$, $\mathbb{F}$-adapted processes $\psi, \phi: [t_0,T] \times \Omega \rightarrow \mathbb{R}^d$, and any $t_0 \leq s<t$ with $t-s \leq 1$, we have 
\begin{align}\label{eq:reg4}
\begin{split}
& \left\|\int_s^t e^{-c(t-r)} \langle \phi_r, f(\psi_r+\widetilde{B}_r) \rangle d r\right\|_{L_m} \\
& \leq C \| f \|_{\mathcal{C}^\gamma} \Big(   \sup_{r \in [s,t]} \| \phi_r \|_{L_m} +  \sup_{r \in [s,t]} \| \phi_r \|_{L_{2m}}  \llbracket \psi \rrbracket_{\mathcal{C}^{1+H(\gamma \wedge 0)}_{[s,t]}L_{2m}} + \llbracket \phi \rrbracket_{\mathcal{C}^{1+H(\gamma \wedge 0)}_{[s,t]}L_m} \Big)  (t-s)^{1+H(\gamma \wedge 0)}  .
\end{split}
\end{align}
\end{prop}

\begin{proof}
    The case when $\gamma \ge 0$ is straightforward because $f$ is bounded. We thus focus on the case $\gamma <0$.
Let $t_0 \leq S \leq T$ with $T-S \leq 1$. For $(s,t) \in \Delta_{[S,T]}$, we define
\begin{align} \label{eq:A-two-}
\begin{split}
  A_{s,t} & :=  \int_s^t e^{-c(T-r)} \langle  \EE^s \phi_{r}, \EE^{s} f(\EE^s \psi_{r}+\widetilde{B}_r) \rangle \, dr,  \\
  \mathcal{A}_{t} & := \int_S^t e^{-c(T-r)}\langle \phi_r, f(\psi_r+\widetilde{B}_r) \rangle \, dr  .
  \end{split}
\end{align}

Assume that the quantities that appear in the right-hand side of \eqref{eq:reg4} are finite, otherwise 
\eqref{eq:reg4} trivially holds. We check the conditions in order to apply the stochastic sewing lemma (Lemma \ref{lemSSL}). Using the Lipschitz norm of $f$, one can check that \eqref{sts3} holds.

We show that \eqref{sts1} and \eqref{sts2} hold true with $\varepsilon_1=1+H\gamma+H(\gamma-1)$ and $\varepsilon_2=1/2+H\gamma$.
Using Lemma \ref{lemreg-B}, we write
\begin{align*}
| A_{s,t} | & \leq  C \int_s^t |\EE^s \phi_{r}| (r-s)^{H  \gamma} \| f \|_{\mathcal{C}^\gamma} \, dr ,
\end{align*}
which implies
\begin{align*}
\| A_{s,t} \|_{L_m} & \leq  C \int_s^t   \| \phi_r\|_{L_m}  (r-s)^{H \gamma} \| f \|_{\mathcal{C}^\gamma} \, dr,
\end{align*}
and consequently, we get
\begin{align}
\| A_{s,t} \|_{L_m} & \leq C \| f \|_{\mathcal{C}^\gamma} \sup_{r \in {[S,T]}}\| \phi_r\|_{L_m}  (t-s)^{1+H \gamma} .  \label{eq:boundAst-two-}
\end{align}
Let $s \leq u \leq t$ and let us now provide a bound on $\EE^s \delta A_{s,u,t}$.  Using the tower property of conditional expectations, we have 
\begin{align*}
\EE^{s} \delta A_{s,u,t}  & = \EE^{s}  \int_u^t  e^{-c(T-r)} \langle \EE^s \phi_{r} , \EE^s f(\EE^s \psi_{r}+\widetilde{B}_r) \rangle \ dr \\
& \quad - \EE^{s}  \int_u^t   e^{-c(T-r)} \langle \EE^u \phi_{r} , \EE^u f(\EE^u \psi_{u}+\widetilde{B}_r) \rangle \ dr .
\\
 & =   \int_u^t   e^{-c(T-r)} \langle \EE^s \phi_{r}, \EE^s  f(\EE^s \psi_{r}+\widetilde{B}_r) -\EE^s f(\EE^u \psi_{r}+\widetilde{B}_r) \rangle dr
 \\
 & \qquad +\EE^s  \int_u^t  e^{-c(T-r)}\langle \EE^s \phi_{r} -\EE^u \phi_{r},  \EE^u f(\EE^u \psi_{r}+\widetilde{B}_r) \rangle \,  dr 
\\ & =: \mathcal{J}_1 + \mathcal{J}_2 .
\end{align*}
Let us first bound $\mathcal{J}_2$. Using Lemma \ref{lemreg-B}~$(i)$, we get
\begin{align*}
 \mathcal{J}_2  & \leq C \| f \|_{\mathcal{C}^\gamma}  \EE^s\int_u^t  |\EE^s \phi_{r}-\EE^u \phi_{r}|    (r-u)^{H \gamma}\, dr. 
 \end{align*}
 Taking $L_m$-norms in the above, using Minkowski's inequality, and the trivial bound 
$$ 
| \EE^s \phi_{r}-\EE^u \phi_{r}|\leq |  \phi_{r}-\EE^s \phi_{r}|+| \phi_{r}-\EE^u \phi_{r}|, 
$$
we get
 \begin{align*}
\|  \mathcal{J}_2 \|_{L_m} 
& \leq C \| f \|_{\mathcal{C}^\gamma}  \llbracket \phi \rrbracket_{\mathcal{C}^{1+H\gamma}_{[S,T]}L_m} (t-s)^{2+2H\gamma} .
\end{align*}
For $\mathcal{J}_1$, using Lemma \ref{lemreg-B}~$(ii)$, we have
\begin{align*}
| \mathcal{J}_1 | & \leq C \int_u^t | \EE^s \phi_{r} |\EE^s| \EE^s\psi_{r}-\EE^u \psi_{r}| (r-u)^{H(\gamma-1)} \| f \|_{\mathcal{C}^\gamma} dr \\
& \leq C\| f \|_{\mathcal{C}^\gamma} \int_u^t  \frac{| \EE^s \phi_{r}| \EE^s| \EE^s \psi_{r}-\EE^u \psi_{r}|}{(r-s)^{1+H\gamma}}  (r-s)^{1+H\gamma}  (r-u)^{H(\gamma-1)} dr .
\end{align*}
 By taking the moments and using the Minkowski and the Cauchy-Schwarz inequalities, we obtain
\begin{align*}
 \| \mathcal{J}_2 \|_{L_m} & \leq C \| f \|_{\mathcal{C}^\gamma} \sup_{r \in [S,T]} \| \phi_{r}\|_{L_{2m}} \llbracket \psi \rrbracket_{\mathcal{C}^{1+H\gamma}_{[S,T]}L_{2m}} (t-s)^{2+H\gamma + H(\gamma-1)} .
\end{align*}

Combining the bounds on $\mathcal{J}_1$ and $\mathcal{J}_2$, we deduce that
\begin{align}\label{eq:bound-deltaA}
\| \EE^s \delta A_{s,u,t} \|_{L_m} \leq C \| f \|_{\mathcal{C}^\gamma} \Big(  \sup_{r \in [S,T]} \| \phi_{r}\|_{L_{2m}} \llbracket \psi \rrbracket_{\mathcal{C}^{1+H\gamma}_{[S,T]}L_{2m}}  + \llbracket \phi \rrbracket_{\mathcal{C}^{1+H\gamma}_{[S,T]}L_m}  \Big) (t-s)^{2+H\gamma + H(\gamma-1)}.
\end{align}
Consequently, we can apply Lemma \ref{lemSSL} and deduce \eqref{eq:reg4}.
\end{proof}

The following proposition will be used when comparing two solutions $X$ and $Y$ and proving that the moments of $X-Y$ decay exponentially.
\begin{prop}\label{prop:reg3}
Let $\gamma  \in(1-1/(2H),1)$ and $m \in[2, \infty)$. There exists a constant $C:= C(m,\gamma,H)>0$ such that for any $c \ge 0$, any $f \in$ $\mathcal{C}^{\infty}$, any $\mathbb{F}$-adapted processes $\psi, \phi: [t_0,T] \times \Omega \rightarrow \mathbb{R}^d$, and any $t_0 \leq s<t$ with $t-s \leq 1$, we have
\begin{align}\label{eq:reg3}
\begin{split}
& \left\|  \int_{s}^{t} e^{-c(t-r)} \langle \psi_r-\phi_r , f(\psi_r+\widetilde{B}_r)-f(\phi_r+\widetilde{B}_r) \rangle dr \right\|_{L_m} \\ & \leq  C  \| f \|_{\mathcal{C}^\gamma}  (1+\llbracket \psi \rrbracket_{\mathcal{C}^{1+H(\gamma \wedge 0)}_{[s,t]} L_{1,\infty}} ) \sup_{r \in [s,t]} \| \psi_r-\phi_r \|_{L_{2m}}^2 (t-s)^{1+H(\gamma-1)} \\ & \quad + C   \| f \|_{\mathcal{C}^\gamma}   \llbracket \psi-\phi \rrbracket_{\mathcal{C}^{1/2}_{[s,t]} L_{2m}}^2  (t-s)^{3/2+H(\gamma-1)} .
\end{split}
\end{align}
\end{prop}
\begin{proof}
Let $t_0 \leq S \leq T$ with $T-S \leq 1$. For $(s,t) \in \Delta_{[S,T]}$, let
\begin{align} \label{eq:Ast}
\begin{split}
  A_{s,t} & :=  \int_s^t e^{-c(T-r)} \langle \EE^s \psi_r-\EE^s \phi_r, \EE^{s} \Big( f(\EE^s \psi_r+\widetilde{B}_r)-f(\EE^s \phi_r+\widetilde{B}_r) \Big) \rangle \, dr  \\
  \mathcal{A}_{t} & := \int_{S}^{t} e^{-c(T-r)} \langle \psi_r-\phi_r , f(\psi_r+\widetilde{B}_r)-f(\phi_r+\widetilde{B}_r) \rangle dr  .
  \end{split}
\end{align}

Assume that the quantities that appear in the right-hand side of \eqref{eq:reg3} are finite, otherwise 
\eqref{eq:reg3} trivially holds. We check the conditions in order to apply the stochastic sewing lemma (Lemma \ref{lemSSL}). Using the Lipschitz norm of $f$, one can check that \eqref{sts3} holds. 

We show that \eqref{sts1} and \eqref{sts2} hold true with $\varepsilon_1=\varepsilon_2=1/2+H(\gamma-1)$.

By Lemma \ref{lemreg-B}~$(ii)$ and Jensen's inequality, we have
\begin{align}
| A_{s,t} | & \leq  C \int_s^t   \EE^s|\psi_r-\phi_r|   \Big| \EE^s \big( f(\EE^s\psi_r+\widetilde{B}_r)-f(\EE^s\phi_r+\widetilde{B}_r) \big) \Big| dr \nonumber \\
& \leq C \int_s^t  \big( \EE^s|\psi_r-\phi_r |^2 \big)  \| f \|_{\mathcal{C}^\gamma} (r-s)^{H(\gamma-1)} dr \nonumber .
\end{align}
Hence,
\begin{align}\label{eq:boundAst3}
\| A_{s,t} \|_{L_m} \leq C  \| f \|_{\mathcal{C}^\gamma} \sup_{r \in [S,T]} \| \psi_r-\phi_r \|_{L_{2m}}^2 (t-s)^{1+H(\gamma-1)}  .
\end{align}
Let $s \leq u \leq t$ and let us now provide the bound on $\EE^s \delta A_{s,u,t}$. By using the tower property of the conditional expectation, we have 
\begin{align*}
\EE^{s} \delta A_{s,u,t}  & =\EE^{s}  \int_u^t  e^{-c(T-r)} \langle \EE^s \psi_r-\EE^s \phi_r,\EE^s (f(\EE^s \psi_r+\widetilde{B}_r)-f(\EE^s \phi_r+\widetilde{B}_r)) -\EE^s (f(\EE^u \psi_r+\widetilde{B}_r) \\ & \quad -f(\EE^u \phi_r+\widetilde{B}_r)) \rangle \, dr \\
& \quad \EE^{s} \int_u^t  e^{-c(T-r)} \langle \EE^s \psi_r-\EE^s \phi_r-(\EE^u \psi_r-\EE^u \phi_r),\EE^u (f(\EE^u \psi_r+\widetilde{B}_r)-f(\EE^u \phi_r+\widetilde{B}_r)) \rangle \, dr  \\
& =: \mathcal{J}_1+\mathcal{J}_2 .
\end{align*}
Let us first bound $\mathcal{J}_2$. Using Lemma \ref{lemreg-B}~$(ii)$, we get
\begin{align*}
 \mathcal{J}_2  & \leq C \| f \|_{\mathcal{C}^\gamma}  \int_u^t  |\EE^s \phi_r-\EE^s \psi_r-(\EE^u \psi_r-\EE^u \phi_r)| |\EE^u \psi_r-\EE^u \phi_r|  (r-u)^{H(\gamma-1)} \, dr  \\
 & \leq C \| f \|_{\mathcal{C}^\gamma}  \int_u^t  \frac{|\EE^s \phi_r-\EE^s \psi_r-(\EE^u \psi_r-\EE^u \phi_r)|}{(r-s)^{1/2}} |\EE^u \psi_r-\EE^u \phi_r|  (r-u)^{H(\gamma-1)} (r-s)^{1/2} \, dr .
 \end{align*}
 By Young's inequality, we get
 \begin{align*}
 \mathcal{J}_2   \leq  C \| f \|_{\mathcal{C}^\gamma} \int_u^t \Big( \frac{|\EE^s \phi_r-\EE^s \psi_r-(\EE^u \psi_r-\EE^u \phi_r)|^2}{r-s} +|\EE^u \psi_r-\EE^u \phi_r|^2 \Big) (r-u)^{H(\gamma-1)} (t-s)^{1/2}.
 \end{align*}
 It follows by taking the moments and using the Minkowski inequality that
 \begin{align}\label{eq:J1}
 \| \mathcal{J}_2 \|_{L_m} \leq C \| f \|_{\mathcal{C}^\gamma}  \Big( \llbracket \psi-\phi \rrbracket_{\mathcal{C}^{1/2}_{[S,T]} L_{2m}}^2 + \sup_{r \in [S,T]} \| \psi_r-\phi_r \|_{L_{2m}}^2 \Big) (t-s)^{3/2+H(\gamma-1)}  .
 \end{align}
For $\mathcal{J}_1$, using Lemma \ref{lemreg-B}~$(iii)$ and conditional Jensen's inequality, we get
\begin{align*}
|\mathcal J_1| & \leq  \int_u^t  e^{-c(T-r)}  \EE^s| \psi_r-\phi_r| \, \EE^s \Big| \EE^u \big( f(\EE^s \psi_r+\widetilde{B}_r)-f(\EE^s \phi_r+\widetilde{B}_r)- f(\EE^u \psi_r+\widetilde{B}_r) \\ & \quad +f(\EE^u \phi_r+\widetilde{B}_r) \big) \Big| \, dr  \\
 & \leq C \|f \|_{\mathcal{C}^\gamma} \int_u^t  \Big( \EE^s| \phi_r- \psi_r |^2 \EE^s |\EE^s \psi_r-\EE^u \psi_r|  (r-u)^{H(\gamma-2)} \\ & \quad + \EE^s | \psi_r-\phi_r|   |\EE^s \psi_r-\EE^s\phi_r -\EE^u \psi_r + \EE^u \phi_r| (r-u)^{H(\gamma-1)} \Big) \, dr \\
 & \leq  C \|f \|_{\mathcal{C}^\gamma} \int_u^t  \Big( \EE^s| \phi_r- \psi_r |^2 \llbracket \psi \rrbracket_{\mathcal{C}^{1+H(\gamma \wedge 0)}_{[S,T]} L_{1,\infty}} (r-s)^{1+H(\gamma \wedge 0)} (r-u)^{H(\gamma-2)} \\ & \quad + \EE^s \big( | \psi_r- \phi_r|   |\EE^s \psi_r-\EE^s\phi_r -\EE^u \psi_r + \EE^u \phi_r| \big) (r-u)^{H(\gamma-1)} \Big) \, dr .
\end{align*}
Hence, by taking the moments, and using Minkowski's, Jensen's and Young's inequalities, we have
\begin{align*}
\| \mathcal J_1\|_{L_m} & \leq C  \| f \|_{\mathcal{C}^\gamma} \llbracket \psi \rrbracket_{\mathcal{C}^{1+H(\gamma \wedge 0)}_{[S,T]} L_{1,\infty}}  \sup_{r \in [S,T]} \| \psi_r-\phi_r \|_{L_{2m}}^2  (t-s)^{2+2 H(\gamma \wedge 0) + H(\gamma-2)}  \\
& \quad + C \| f \|_{\mathcal{C}^\gamma} \Big( \llbracket \psi-\phi \rrbracket_{\mathcal{C}^{1/2}_{[S,T]} L_{2m}}^2 + \sup_{r \in [S,T]} \| \psi_r-\phi_r \|_{L_{2m}}^2 \Big) (t-s)^{3/2+H(\gamma-1)} .
\end{align*}
Combining the bounds on $\mathcal{J}_1$ and $\mathcal{J}_2$, we deduce that
\begin{align}\label{eq:bound-deltaA-2}
\begin{split}
\| \EE^s \delta A_{s,u,t} \|_{L_m} & \leq C \| f \|_{\mathcal{C}^\gamma}  \Big( \llbracket \psi-\phi \rrbracket_{\mathcal{C}^{1/2}_{[S,T]} L_{2m}}^2 \\ & \quad + (1+\llbracket \psi \rrbracket_{\mathcal{C}^{1+H(\gamma \wedge 0)}_{[S,T]} L_{1,\infty}} )\sup_{r \in [S,T]} \| \psi_r-\phi_r \|_{L_{2m}}^2  \Big) (t-s)^{3/2+H(\gamma-1)} .
\end{split}
\end{align}
Combining \eqref{eq:boundAst3} and \eqref{eq:bound-deltaA-2}, we can apply Lemma \ref{lemSSL} and deduce \eqref{eq:reg3}.
\end{proof}

\subsection{A priori estimates}\label{sec:apriori}
We establish in this section estimates on the regularity of the solution when $b$ is a smooth drift, then we state stability estimates of solutions with respect to the drift and the initial condition. Throughout the whole section, we fix a filtered probability space $(\Omega,\mathcal{F},\mathbb{F},\mathbb{P})$, an 
$\mathbb{F}$-Wiener process $W$, and an fBm given via $W$ by \eqref{eq:MVN}. Whenever a solution is mentioned, it is meant to be with respect to this probability space and Wiener process.

\paragraph{Regularity estimates for smooth drift.}
We assume that $b \in \mathcal{C}^\infty$. For any $x \in \R^d$ and $t_0 \in \R$, we write $(X_t)_{t \ge t_0}$ for the solution of \eqref{eq:sde} starting from $x$ at time $t_0$, and prove estimates on $X$ that depend only on the Besov-H\"older norm of $b$.

The following proposition gives the Besov-H\"older regularity of the solution $X$.
\begin{prop}\label{prop:regZ}
Let $\gamma \in (1/2-1/2H,1)$ and $m \in [2,\infty)$.
Let $F,b$ satisfy Assumptions \ref{assumpF} and \ref{assumpb}, and further assume that $b \in \mathcal{C}^\infty$. There exists a constant $C:=C(m,\gamma,H,\kappa,\Xi)$ such that for any $x \in \R^d$, any $t_0 \in \R$, and any $t_0 \leq s<t$ with $t-s \leq 1$, we have 
\begin{align}\label{eq:regZ}
\llbracket X-B \rrbracket_{\mathcal{C}^{1+H(\gamma \wedge 0)}_{[s,t]} L_m} \leq 2 [X-B]_{\mathcal{C}^{1+H(\gamma \wedge 0)}_{[s,t]} L_m} \leq  C \Big( \sup_{r \in [s,t]} \|X_r\|_{L_m}+1 \Big).
\end{align}
\end{prop}
\begin{proof}
Recall that the first inequality is always satisfied by the definition of the semi-norms in \eqref{eq:semi-norm} and \eqref{eq:semi-norm-2}.
Let $t_0 \leq S<T$ with $T-S \leq 1$ and let $(s,t) \in \Delta_{[S,T]}$, then
\begin{align*}
X_t-B_t - (X_s-B_s) = \int_s^t F(X_r)dr + \int_s^t b(X_r) dr .
\end{align*}

First, we prove that $[X-B]_{\mathcal{C}^{1+H(\gamma \wedge 0)}_{[S,T]} L_m}$ is finite. Using the growth of $F$ and that $b$ is bounded, we have
\begin{align*}
\| X_t-B_t - (X_s-B_s) \|_{L_m} & \leq C \int_s^t (1+\| X_r\|_{L_m} ) dr +  \| b \|_{\infty} (t-s) .
\end{align*}
Moreover, since $b$ is smooth, $F+b$ still satisfies \eqref{eq:Fdissipative} and \eqref{eq:growth}, hence the moments of $X$ are bounded uniformly in time. Dividing by $(t-s)^{1+H(\gamma \wedge 0)}$ and taking the supremum of $(s,t)$ over $\Delta_{[S,T]}$, it follows that $[X-B]_{\mathcal{C}^{1+H(\gamma \wedge 0)}_{[S,T]} L_m}$ is indeed finite.

We can now prove \eqref{eq:regZ}.
Using the growth of $F$ again, it follows that
\begin{align*}
\| X_t-B_t - (X_s-B_s) \| _{L_m} & \leq C \Big( \int_s^t(1+ \|X_r\|_{L_m}) dr + \| \int_s^t b(X_r) dr \|_{L_m} \Big) .
\end{align*}
Applying Proposition \ref{prop:reg1} with $\tau=1+H(\gamma \wedge 0)$, $\psi = X-\widetilde{B}^{t_0}$ and noticing by \eqref{eq:compar-norms} that
\begin{align}\label{eq:link-holdernorms}
\llbracket X-\widetilde{B}^{t_0} \rrbracket_{\mathcal{C}^{1+H(\gamma \wedge 0)}_{[S,T]} L_m} = \llbracket X-B \rrbracket_{\mathcal{C}^{1+H(\gamma \wedge 0)}_{[S,T]} L_m}  \leq 2 [X-B]_{\mathcal{C}^{1+H(\gamma \wedge 0)}_{[S,T]} L_m},
\end{align}
we get 
\begin{align*}
\| X_t-B_t - (X_s-B_s) \|_{L_m} 
& \leq C \Big( \sup_{r \in [S,T]} \|X_r\|_{L_m} (t-s)+ (t-s)  \\ & \quad +  [X-B]_{\mathcal{C}^{1+H\gamma}_{[S,T]} L_m}(t-s)^{2+H(\gamma -1)+H (\gamma \wedge 0)} +    (t-s)^{1+H (\gamma \wedge 0)} \Big)  .
\end{align*}
Dividing by $(t-s)^{1+H (\gamma \wedge 0)}$ and taking the supremum over $(s,t) \in \Delta_{[S,T]}$, we get
\begin{align*}
 [X-B]_{\mathcal{C}^{1+H (\gamma \wedge 0)}_{[S,T]} L_m} & \leq C \Big( [X-B]_{\mathcal{C}^{1+H(\gamma \wedge 0)}_{[S,T]} L_{m,\infty}} (T-S)^{1+H(\gamma-1)} \\ & \quad +  \sup_{r \in [S,T]} \|X_r\|_{L_m} (T-S)^{-H (\gamma \wedge 0)}+1 \Big)   .
\end{align*}
Let $\ell := \left( 2 C  \right)^{-1/(1+H(\gamma-1))}$. Then for $T-S \leq \ell$, we have $C (T-S) \leq 1/2$. It follows that for $T-S \leq \ell$
\begin{align*}
[X-B]_{\mathcal{C}^{1+H(\gamma \wedge 0)}_{[S,T]} L_m} & \leq C \Big( \sup_{r \in [S,T]} \|X_r\|_{L_m} (T-S)^{-H (\gamma \wedge 0)}+1 \Big).
\end{align*}
Applying the above inequality over a finite number of intervals of size $\ell$, and using the sub-additivity of the H\"older semi-norm, we get the desired result.
\end{proof}

In the next proposition, we state how the moments of $X$ behave in time, the result implies in particular that these moments are bounded uniformly in time if $\kappa_1>0$.
\begin{prop}\label{prop:unif-moments}
Let $\gamma \in (1/2-1/(2H),1)$ and  $m \in [2,\infty)$.
Let $F,b$ satisfying Assumptions \ref{assumpF} and \ref{assumpb}, and further assume that $b \in \mathcal{C}^\infty$. There exists a constant $C := C(m,\gamma,H,x,\kappa,\Xi)$, such that for any $x \in \R^d$, any $T \in \R$, any $t_0 < T$, we have
\begin{align*}
\sup_{r \in [t_0,T]} \|X_r\|_{L_m} \leq C \Big(1+ (T-t_0) \wedge \frac{1-e^{-\kappa_1(T-t_0)}}{\kappa_1}  \Big) .
\end{align*}
\end{prop}
\begin{proof}
Let $U$ denote the Ornstein-Uhlenbeck process defined in \eqref{eq:U} starting at $U_{t_0}=x$. In particular, $U$ is a Gaussian process and the variance of $U_t$ is bounded uniformly in $t$. This implies that
\begin{align}\label{eq:unif-U}
\sup_{t \in [t_0,\infty)} \| U_t \|_{L_m} \leq C.
\end{align}
Moreover, for any $t_0< s <t$ with $t-s \leq 1$, we have $U_t-B_t-(U_s-B_s)=-\int_s^t U_r dr$, which implies that
\begin{align}\label{eq:holder-U}
[U-B]_{\mathcal{C}^{1}_{[s,t]} L_m} \leq C. 
\end{align}
Let $t \ge t_0$, since the processes $X-B$ and $U-B$ are almost surely differentiable, we write

\begin{align*}
\frac{d}{d t}\left|X_t-U_t\right|^{2} & = 2 \langle X_t-U_t,F\left(X_t\right)+U_t+b\left(X_t\right)\rangle  \\
& =2\langle X_t-U_t,F\left(X_t\right)-F(U_t) \rangle +2 \langle X_t-U_t, F(U_t)+U_t\rangle + \langle X_t-U_t, b\left(X_t\right)\rangle .
\end{align*}
Using that $F$ has linear growth, \eqref{eq:Fdissipative} and Young's inequality, it follows that
\begin{align*}
\frac{d}{d t}\left|X_t-U_t\right|^{2} \leq & -2\kappa_1 |X_t-U_t|^2+2\kappa_2+ 2\langle X_t-U_t, F(U_t)+U_t\rangle +2 \langle X_t-U_t, b\left(X_t\right)\rangle \\
\leq & -\kappa_1 |X_t-U_t|^{2}+2\kappa_2  +C(\kappa_3+1)\left(|U_t|^{2}+1\right)+2\langle X_t-U_t, b(X_t)\rangle \\
\leq & -\kappa_1\left|X_t-U_t\right|^{2} +C\left(|U_t|^{2}+1\right)+2\langle X_t-U_t, b\left(X_t\right)\rangle  \\
& \leq -\kappa_1 |X_t-U_t|^2 + C(|U_t|^{2}+1) +C \langle X_t-U_t, b(X_t)\rangle .
\end{align*}
Using a Gr\"onwall-type argument and assuming that $X$ and $U$ start at the same point, we get that
\begin{align*}
|X_t-U_t|^2 & \leq C \Big( 1+ \int_{t_0}^t e^{-\kappa_1(t-r)} |U_r|^2 dr + \int_{t_0}^t e^{-\kappa_1(t-r)} \langle X_r-U_r, b\left(X_r\right)\rangle dr \Big) .
\end{align*}
Let $\delta \in (0,1)$ to be determined below. Letting $(t_k)_{k \in \llfloor 0, N \rrfloor}$ be a partition of $[{t_0},t]$ such that $t_{k+1}-t_k =\delta < 1$ and $N= \lfloor (t-t_0)/\delta \rfloor$. By using \eqref{eq:unif-U}, we have
\begin{align}\label{eq:ou-comp}
\begin{split}
\| X_t-U_t \|_{L_{2m}}^2 & \leq C  \Big( 1+ C_0(\kappa_1,t-t_0)  \\ & \quad + \sum_{k=0}^{N-1} e^{-\kappa_1 (t-t_{k+1})} \| \int_{t_k}^{t_{k+1}} e^{-\kappa_1 (t_{k+1}-r)} \langle X_r-U_r, b\left(X_r\right)\rangle dr \|_{L_m} \\ & + \| \int_{t_N}^t e^{-\kappa_1 (t-r)} \langle X_r-U_r, b\left(X_r\right)\rangle dr \Big)  ,
\end{split}
\end{align}
where
\begin{align}\label{eq:defC0}
    C_0(\kappa_1,t-t_0) := \int_{t_0}^{t} e^{-\kappa_1 (t-r)} dr= \begin{cases}
    \kappa_1^{-1}(1-e^{-\kappa_1 (t-t_0)}) \, &\text{if  $\kappa_1 \neq 0$} \\
    t-t_0 \, &\text{if $\kappa_1=0$} .
    \end{cases}
\end{align}
Let us write $t_{N+1} = t$. Applying Proposition \ref{prop:reg4} with $\phi=X-U$, $\psi=X-\widetilde{B}^{t_0}$, noticing that 
$$ \llbracket X-\widetilde{B}^{t_0} \rrbracket_{\mathcal{C}^{1+H(\gamma \wedge 0)}_{[t_{k},t_{k+1}]}L_{2m}}  = \llbracket X-B \rrbracket_{\mathcal{C}^{1+H(\gamma \wedge 0)}_{[t_{k},t_{k+1}]}L_{2m}},$$ 
and using \eqref{eq:compar-norms}, we get for any $k \in \llfloor 0, N \rrfloor$,
\begin{align}\label{eq:ou-comp2}
\begin{split}
& \| \int_{t_k}^{t_{k+1}} e^{-\kappa_1 (t_{k+1}-r)} \langle X_r-U_r, b\left(X_r\right)\rangle dr \|_{L_m}  \\ & \leq  C \| b\|_{\mathcal{C}^\gamma}   \Big( \sup_{r \in [t_0,t]} \| X_r-U_r \|_{L_m} \delta^{1+H(\gamma \wedge 0)} \\ & \quad +  \Big( \sup_{r \in [t_0,t]} \| X_r-U_r \|_{L_{2m}} [ X-B ]_{\mathcal{C}^{1+H(\gamma \wedge 0) }_{[t_{0},t]}L_{2m}}  + [ X-U]_{\mathcal{C}^{1+H(\gamma \wedge 0) }_{[t_{0},t]}L_m} \Big)  \delta^{2+2H(\gamma \wedge 0)} \Big) . 
\end{split}
\end{align}
By Proposition \ref{prop:regZ}, we have that 
\begin{align}\label{eq:resultprop-holder}
[ X-B ]_{\mathcal{C}^{1+H(\gamma \wedge 0) }_{[t_{0},t]}L_{2m}}  \leq C (1+\sup_{r \in [t_0,t]}\|X\|_{L_{2m}} ) .
\end{align} 
Moreover, writing $X-U=X-B+B-U$ and using \eqref{eq:regZ} and \eqref{eq:holder-U}, one gets
\begin{align}\label{eq:resultprop-holderU}
[X-U]_{\mathcal{C}^{1+H(\gamma \wedge 0) }_{[t_{0},t]}L_{m}} \leq C (1+\sup_{r \in [t_0,t]}\|X\|_{L_m} + [U-B]_{\mathcal{C}^{1+H(\gamma \wedge 0) }_{[t_{0},t]}L_{m}} )\leq C(1+\sup_{r \in [t_0,t]}\|X\|_{L_m}).
\end{align}
Notice that $\sum_{k=0}^{N-1} e^{-\kappa_1 (t-t_{k+1})} \lesssim \delta^{-1} C_0(\kappa_1,T-t_0)$. Hence, plugging \eqref{eq:ou-comp2} back in \eqref{eq:ou-comp} and using \eqref{eq:unif-U}, \eqref{eq:resultprop-holder} and \eqref{eq:resultprop-holderU}, we get
\begin{align*}
\| X_t-U_t \|_{L_{2m}}^2 & \leq C \big(1+C_0(\kappa_1,T-t_0) \big)+ C C_0(\kappa_1,T-t_0)  \Big( \sup_{r \in [t_0,t)}\| X_r-U_r \|_{L_{2m}} \delta^{H(\gamma \wedge 0)} \\ & \quad + \Big( \sup_{r \in [t_0,t]}\| X_r-U_r\|_{L_{2m}}   (1+\sup_{r \in [t_0,t]}\|X_r\|_{L_{2m}} ) +1+\sup_{r \in [t_0,t]}\|X_r\|_{L_m}  \Big)  \delta^{1+2H(\gamma \wedge 0)} \Big).
\end{align*}
By taking the supremum over $t$ in $[t_0,T]$, and using that $\|X_r\|_{L_m} \leq \|X_r\|_{L_{2m}} \leq \|X_r-U_r\|_{L_{2m}}+C$ (by \eqref{eq:unif-U}), we get
\begin{align*}
\sup_{r \in [t_0,T]} \| X_r-U_r \|_{L_{2m}}^2 & \leq  C (1+C_0(\kappa_1,T-t_0))  \Big(1+\sup_{r \in [t_0,T]}\| X_r-U_r \|_{L_{2m}} \delta^{H(\gamma \wedge 0)}  \\ & \quad +  \sup_{r \in [t_0,T]}\| X_r-U_r\|_{L_{2m}}^{2} \delta^{1+2H(\gamma \wedge 0)} \Big).
\end{align*}
Recall that $1+2H(\gamma \wedge 0)>0$. By choosing $\delta =(2C(1+C_0(\kappa_1,T-t_0))  )^{-\frac{1}{1+2H(\gamma \wedge 0)}}$, so that $CC_0(\kappa_1,T-t_0)  \delta^{1+2H(\gamma \wedge 0)} < 1/2$, we get that
\begin{align*}
\sup_{r \in [t_0,T]} \| X_r-U_r \|_{L_{2m}}^2 & \leq C (1+C_0(\kappa_1,T-t_0))  \Big(1+\sup_{r \in [t_0,T]}\| X_r-U_r \|_{L_{2m}}  \delta^{H(\gamma \wedge 0)} \Big).
\end{align*}
By using Young's inequality, rearranging, \eqref{eq:unif-U} and the definition of $C_0$, we conclude the proof.
\end{proof}

We prove in the following proposition another estimate for the solution $X$ when $F$ is Lipschitz. This will be useful in the proof of uniqueness.
\begin{prop}\label{prop:regZ-infty}
Let $\gamma \in (1/2-1/(2H),1)$.
Let $b$ satisfy Assumption \ref{assumpb}, and further assume that $b \in \mathcal{C}^\infty$. 
Let $F$ satisfy Assumption \ref{Flip}. There exists a constant $C:=C(m,\gamma,H,L_F,\Xi)$ such that for any $x \in \R^d$, any $t_0 \in \R$, any $t_0 \leq S<T$ with $T-S \leq 1$, we have
\begin{align}\label{eq:regZ-infty}
\llbracket X-B \rrbracket_{\mathcal{C}^{1+H(\gamma \wedge 0)}_{[S,T]} L_{m,\infty}} \leq  C .
\end{align}
\end{prop}
\begin{proof}
Let $(s,t) \in \Delta_{[S,T]}$, then
\begin{align*}
X_t-B_t - (\EE^s X_t-\EE^s B_t) = \int_s^t F(X_r)-\EE^s F(X_{r}) dr + \int_s^t b(X_r)-\EE^s b(X_r) dr .
\end{align*}
\paragraph{Proving that $\llbracket X-B \rrbracket_{\mathcal{C}^{1+H(\gamma \wedge 0)}_{[S,T]} L_{m,\infty}}$ is finite.}
First, let us show that the quantity that appears on the left-hand side of \eqref{eq:regZ-infty} is finite. Using that $F$ is Lipschitz continuous and that $b$ is bounded, we have
\begin{align*}
| X_t-B_t - (\EE^s X_t-\EE^s B_t) |& \leq C \int_s^t  |X_r-\EE^{s} X_r| dr +  \| b \|_{\infty} (t-s) \\
& \leq C \int_s^t  |X_r-B_r-\EE^{s} X_r+\EE^sB_r| dr +\int_s^t |B_r-\EE^s B_r| dr+   \| b \|_{\infty} (t-s)
\end{align*}
Since $B_r-\EE^s B_r$ is a Gaussian variable independent of $\mathcal{F}_s$ with variance proportional to $(r-s)^{2H}$, it follows that
\begin{align*}
\EE^s| X_t-B_t - (\EE^s X_t-\EE^s B_t) | & \leq  C \int_s^t  \EE^s |X_r-B_r-\EE^{s} X_r+\EE^sB_r| dr +(t-s)^{1+H}+   \| b \|_{\infty} (t-s) .
\end{align*}
Gr\"onwall's lemma implies that
\begin{align*}
\EE^s| X_t-B_t - (\EE^s X_t-\EE^s B_t) | \leq C(T) (\| b \|_{\infty}+1)  (t-s) .
\end{align*}
 Dividing by $(t-s)^{1+H(\gamma \wedge 0)}$ and taking the supremum of $(s,t)$ over $\Delta_{[S,T]}$, we conclude that $\llbracket X-B \rrbracket_{\mathcal{C}^{1+H(\gamma \wedge 0)}_{[S,T]} L_{m,\infty}}$ is indeed finite. 
\paragraph{Proving \eqref{eq:regZ-infty}.}
Using similar arguments as above, one can show that
\begin{align*}
\| X_t-B_t - \EE^s (X_t-B_t) \|_{L_{m,\infty}^{\mathcal{F}_s}} & \leq  C \int_s^t  \| X_t-B_t - \EE^s (X_t-B_t) \|_{L_{m,\infty}^{\mathcal{F}_s}}  dr +(t-s)^{1+H} \\ & \quad + \Big\|  \int_s^t  b(X_r)- \EE^s b(X_r) dr  \Big\|_{L^{\mathcal{F}_s}_{m,\infty}} \\
& \leq  C \int_s^t \| X_t-B_t - \EE^s (X_t-B_t) \|_{L_{m,\infty}^{\mathcal{F}_s}}  dr +(t-s)^{1+H} \\ & \quad + 2 \left\| \int_s^t  b(X_r) dr \right\|_{L_{m,\infty}^{\mathcal{F}_s}}.
\end{align*}
Applying Proposition \ref{prop:reg1} with $\tau=1+H(\gamma \wedge 0)$, $(\psi_t)_t=(X_t-\widetilde{B}_{t-t_0})_{t \ge t_0}$ and using \eqref{eq:link-holdernorms}, we get 
\begin{align*}
\| X_t-B_t - \EE^s (X_t-B_t) \|_{L_{m,\infty}^{\mathcal{F}_s}} 
& \leq C \Big( \sup_{r \in [S,T]} \| X_r-B_r - \EE^s (X_r-B_r)\|_{L_{m,\infty}^{\mathcal{F}_s}} (t-s) \\ & \quad + (t-s)+  \llbracket X-B \rrbracket_{\mathcal{C}^{1+H(\gamma \wedge 0)}_{[S,T]} L_{m,\infty}}(t-s)^{2+H(\gamma -1)+H (\gamma \wedge 0)} \\ & \quad +    (t-s)^{1+H (\gamma \wedge 0)} \Big)  .
\end{align*}
Divide by $(t-s)^{1+H (\gamma \wedge 0)}$ and take the supremum over $(s,t) \in \Delta_{[S,T]}$ to get
\begin{align*}
 \llbracket X-B \rrbracket_{\mathcal{C}^{1+H (\gamma \wedge 0)}_{[S,T]} L_{m,\infty}} & \leq C \Big( \llbracket X-B \rrbracket_{\mathcal{C}^{1+H(\gamma \wedge 0)}_{[S,T]} L_{m,\infty}} (T-S)^{1+H(\gamma-1)}  +(T-S)^{-H  (\gamma \wedge 0)} + 1 \Big).
\end{align*}
Let $\ell := \left( 2 C \right)^{-1/(1+H(\gamma-1))}$, then for $T-S \leq \ell$, we have $C (T-S) \leq 1/2$. It follows that for $T-S \leq \ell$
\begin{align*}
\llbracket X-B \rrbracket_{\mathcal{C}^{1+H(\gamma \wedge 0)}_{[S,T]} L_{m,\infty}} & \leq C .
\end{align*}
Repeating the arguments over a finite number of intervals of size $\ell$, and using the sub-additivity of the H\"older-like semi-norm we get the desired result.
\end{proof}

We assume now that $b \in \mathcal{C}^\gamma$ and that $F$ is Lipschitz continuous and satisfies \eqref{eq:Fdissipative} with $\kappa_2=0$. For any $m \in [2,\infty)$, we define $\mathcal{V}(m)$ to be the class
\begin{align}\label{eq:classVm}
 \mathcal{V}(m) = \{ Y: \, [Y-B ]_{\mathcal{C}^{1/2}_{[t_0,T]} L_{m}} < \infty \} .
\end{align}
In particular $\mathcal{V}(2)$ coincides with the class $\mathcal{V}$ defined in \eqref{eq:class-uniqueness}. 
The rest of this section is divided into two parts, we first study the stability of the SDE with respect to the drift and then with respect to the initial condition. The stability estimates will be given for solutions that belong to the class $\mathcal{V}(m)$ for some fixed $m$.

\paragraph{Stability with respect to the singular drift.}

The following proposition gives an upper bound on the $1/2$-H\"older norm of the difference of two solutions with different singular drifts.

\begin{prop}\label{prop:stab-drift}
Let $\gamma \in (1-1/(2H),1)$, $m \in [2,\infty)$, $T \in \R$ and $t_0<T$.
Let $b,g$ satisfy Assumption \ref{assumpb} with the same constant $\Xi$ and let $F$ satisfy Assumption \ref{Flip}. There exist positive constants $C_1:=C_1(m,\gamma,H,L_F,\Xi,T-t_0)$ and $C_2=C_2(m,\gamma,H,\Xi)$ such that for any
$x \in \R^d$, any two solutions $(X_t)_{t \in [t_0,T]} , (Y_t)_{t \in [t_0,T]}$ to the SDE \eqref{eq:sde} with the same initial condition $x$ and respective drifts $b,g$ such that $X$ and $Y$ belong to the class $\mathcal{V}(m)$, we have
\begin{align*}
 [ X-Y ]_{\mathcal{C}^{1/2}_{[t_0,T]} L_m}  
& \leq   C_1 (1+\llbracket X-B \rrbracket_{\mathcal{C}^{1+H(\gamma \wedge 0)}_{[t_0,T]} L_{m,\infty}})^{C_2} \|b-g\|_{\mathcal{C}^{\gamma-1}}  .
\end{align*}
\end{prop}
\begin{proof}
Assume without any loss of generality that $\llbracket X-B \rrbracket_{\mathcal{C}^{1+H(\gamma \wedge 0)}_{[t_0,T]} L_{m,\infty}}<+\infty$.
Let $K^X, K^Y$ denote the drift processes associated to $Y,X$ given by Definition \ref{defsol-SDE}. Let $b,g$ in $\mathcal{C}^{\gamma+}$ and let
$(b^n)_{n \in \N}, (g^n)_{n \in \N}$ be two sequences converging respectively to $b,g$ in $\mathcal{C}^{\gamma}$ and such that $\sup_{n \in \N} \|g^n\|_{\mathcal{C}^{\gamma}} \leq \Xi $, and such that $K^X$ and $K^Y$ are the almost sure limit of $\int g^n(Y)$ and $\int b^n(X)$ over $[t_0,T]$ respectively.
Let $t_0 \leq \tilde{s} <\tilde{t}$ with $\tilde{t}-\tilde{s}\leq 1$, and let $(s,t) \in \Delta_{[\tilde{s},\tilde{t}]}$ we have
\begin{align*}
Y_t-Y_s-(X_t-X_s)=\int_s^t F(Y_r)-F(X_r) dr + K^Y_t-K^Y_s-K^X_t+K^X_s .
\end{align*}
Using that $F$ is Lipschitz, taking the moments and using Fatou's lemma, we get 
\begin{align}
\| Y_t-Y_s-(X_t-X_s) \|_{L_m} & \leq C \int_s^t \| Y_r-X_r\|_{L_m} dr +\lim_{n \rightarrow \infty}\| \int_s^t g^n(Y_r)-b^n(X_r) dr  \|_{L_m}  \nonumber \\
\begin{split}\label{eq:y-yn}
& \leq  C \Big( [ Y-X]_{\mathcal{C}^{1/2}_{[\tilde{s},\tilde{t}]} L_m} + \| Y_{\tilde{s}}-X_{\tilde{s}} \|_{L^m} \Big) (t-s) \\ & \quad +\limsup_{n \rightarrow \infty} \| \int_s^t g^n(Y_r)-g^n(X_r) dr  \|_{L_m}  \\ & \quad + \limsup_{n \rightarrow \infty}  \| \int_s^t (g^n-b^n)(X_r) dr  \|_{L_m} .
\end{split}
\end{align}
Applying Proposition \ref{prop:reg1} with $q=m$, $f=g^n-b^n$, $\tau=1+H(\gamma \wedge 0)$ and $\psi=Y_t-\widetilde{B}^{t_0}$, we get that
\begin{align}\label{eq:bm-bn}
\| \int_s^t (g^n-b^n)(X_r) dr  \|_{L_m} & \leq C (1+\llbracket X-B \rrbracket_{\mathcal{C}^{1+H(\gamma \wedge 0)}_{[\tilde{s},\tilde{t}]} L_m}) \lim_{n \rightarrow \infty} \|g^n-b^n\|_{\mathcal{C}^{\gamma-1}} (t-s)^{1+H(\gamma -1)} .
\end{align}
Applying now Proposition \ref{prop:reg2} with $f=g^n$, $\psi=X-\widetilde{B}^{t_0}$, $\phi=Y-\widetilde{B}^{t_0}$ and $\tau=1/2$, we get that
\begin{align}\label{eq:bm}
\begin{split}
\| \int_s^t g^n(Y_r)-g^n(X_r) dr  \|_{L_m} & \leq  C  \Big(  \llbracket Y-X \rrbracket_{\mathcal{C}^{1/2}_{[\tilde{s},\tilde{t}]} L_m} (t-s)^{3/2+H(\gamma-1)} \\ & \quad + (1+\llbracket X-B \rrbracket_{\mathcal{C}^{1+H(\gamma \wedge 0)}_{[s,t]} L_{m,\infty}})\| Y_{\tilde{s}}-X_{\tilde{s}}\|_{L_m}  (\tilde{t}-\tilde{s})^{1+H(\gamma-1 )} \Big) .
\end{split}
\end{align}
Plugging \eqref{eq:bm} and \eqref{eq:bm-bn} back in \eqref{eq:y-yn}, using that $\llbracket X-B \rrbracket_{\mathcal{C}^{1+H(\gamma \wedge 0)}_{[\tilde{s},\tilde{t}]} L_m}  \leq  \llbracket X-B \rrbracket_{\mathcal{C}^{1+H(\gamma \wedge 0)}_{[\tilde{s},\tilde{t}]} L_{m,\infty}}  $ and \eqref{eq:compar-norms}, dividing by $(t-s)^{1/2}$ and taking the supremum over $\Delta_{[\tilde{s},\tilde{t}]}$, we conclude that
\begin{align}\label{eq:stab-inter}
\begin{split}
[ Y-X ]_{\mathcal{C}^{1/2}_{[\tilde{s},\tilde{t}]} L_m}  & \leq C \Big( (1+\llbracket X-B \rrbracket_{\mathcal{C}^{1+H(\gamma \wedge 0)}_{[t_0,T]} L_{m,\infty}}) (\| Y_{\tilde{s}}-X_{\tilde{s}}\|_{L_m}+\|g-b\|_{\mathcal{C}^{\gamma-1}} ) \\ & \quad +  \,  [ Y-X ]_{\mathcal{C}^{1/2}_{[\tilde{s},\tilde{t}]} L_m} (\tilde{t}-\tilde{s})^{1+H(\gamma-1)}  \Big) .
\end{split}
\end{align}
Since $X,Y$ are in the class $\mathcal{V}(m)$, we have that $[ Y-X ]_{\mathcal{C}^{1/2}_{[\tilde{s},\tilde{t}]} L_m} < +\infty$. Therefore, for $\tilde{t}-\tilde{s} \leq \ell := (2C)^{-1/(1+H(\gamma-1) )}$, we get
\begin{align*}
[ X-Y ]_{\mathcal{C}^{1/2}_{[\tilde{s},\tilde{t}]} L_m} & \leq C (1+\llbracket X-B \rrbracket_{\mathcal{C}^{1+H(\gamma \wedge 0)}_{[t_0,T]} L_{m,\infty}}) (\| X_{\tilde{s}}-Y_{\tilde{s}}\|_{L_m} +  \|b-g\|_{\mathcal{C}^{\gamma-1}} )  .
\end{align*}
Splitting $[t_0,T]$ into $\lceil \frac{1}{\ell} \rceil$ small consecutive intervals of size at most $\ell$, iterating the above and using the sub-additivity of the H\"older semi-norm, we obtain the claimed estimate.
\end{proof}

\paragraph{Quantitative stability with respect to the initial condition.}
We now investigate the stability with respect to the initial condition. It will be important to explicit the dependence in time of the estimates as they will be used to analyse the long-time behaviour of solutions.
First, we need an intermediary lemma on the H\"older regularity of the difference of any two solutions.
\begin{lemma}\label{prop:regZ-W-lip}
Let $\gamma \in (1-1/(2H),1)$ and $m \in [2,\infty)$. Let $b$ satisfy Assumption \ref{assumpb} and further assume that $b \in \mathcal{C}^\infty$.
Let $F$ satisfy Assumption \ref{Flip}. 
There exists a constant $C:= C(m,\gamma,H,L_F,\Xi)$ such that for any
$x,y \in \R^d$, any $T \in \R$, any $t_0<T$, any two solutions $(X_t)_{t \in [t_0,T]} , (Y_t)_{t \in [t_0,T]}$ to the SDE \eqref{eq:sde} with the same drift $b$ and respective initial conditions $x$ and $y$ such that $X$ and $Y$ belong to the class $\mathcal{V}(m)$, and any $t_0\leq s < t$ with $t-s\leq 1$, we have
\begin{align*}
 [ X-Y ]_{\mathcal{C}^{1/2}_{[s,t]} L_m}  \leq  C (1+\llbracket X-B \rrbracket_{\mathcal{C}^{1+H(\gamma \wedge 0)}_{[s,t]} L_{m,\infty}})  \sup_{r \in [s,t]} \| X_r-Y_r \|_{L_m} .   
\end{align*}
\end{lemma}
\begin{proof}
Let $t_0 \leq  s<  t$ such that $ t- s \leq 1$. Assume without any loss of generality that $\llbracket X-B \rrbracket_{\mathcal{C}^{1+H(\gamma \wedge 0)}_{[ s,  t]} L_{m,\infty}} <+\infty$. Let $(\tilde s,\tilde t) \in \Delta_{[s,t]}$. One can follow the exact same lines as in the proof of Proposition \ref{prop:stab-drift}, taking $b^n=b=g=g^n$, to arrive at \eqref{eq:stab-inter}. Then, letting $\ell := (2C )^{-1/(1+H(\gamma-1))}$ and using that $X,Y$ are in the class $\mathcal{V}(m)$ and $ \| X_{\tilde s} - Y_{\tilde s} \|_{L_m} \leq \sup_{r \in [s,t]} \| X_{r} - Y_{r} \|_{L_m}$, we get for $t-s \leq \ell$
\begin{align}\label{eq:stab-inter2}
[X-Y]_{\mathcal{C}^{1/2}_{[ s, t]}L_m} & \leq C(1+\llbracket X-B \rrbracket_{\mathcal{C}^{1+H(\gamma \wedge 0)}_{[s,t]} L_{m,\infty}}) \sup_{r \in [s,t]} \| X_r-Y_r \|_{L_m}  .
\end{align}
Repeating this argument and using the sub-additivity of the H\"older semi-norm, we conclude that \eqref{eq:stab-inter2} holds for $t-s \leq 1$.
\end{proof}

We can now state a stability estimate with respect to the initial condition. This result will be useful for uniqueness of solutions (Section \ref{sec:strongexistence}) and the long-time behaviour (Section \ref{sec:longtime}).
\begin{prop}\label{prop:stab-lip}
Let $\gamma \in (1-1/(2H),1)$ and $m \in [2,\infty)$. Let $b$ satisfy Assumption \ref{assumpb}.
Let $F$ satisfy Assumption \ref{Flip}, and further assume that $F$ satisfies \eqref{eq:Fdissipative} with $\kappa_2=0$. 
There exists a universal constant $C$ and a positive constant $\mathbf{M}:=\mathbf{M}(m,\gamma,H,L_F,\Xi,C_0)$ such that for any $x,y \in \R^d$, any $T \in \R$ any $t_0<T$, any two solutions $(X_t)_{t \in [t_0,T]} , (Y_t)_{t \in [t_0,T]}$ the SDE \eqref{eq:sde} with the same drift $b$ and respective initial conditions $x$ and $y$ such that $X$ and $Y$ belong to the class $\mathcal{V}(m)$ and $X$ satisfies: there exists a constant $C_0$
such that for any $t_0<s<t$ with $t-s \leq 1$
\begin{align}\label{eq:inf}
\llbracket X-B \rrbracket_{\mathcal{C}^{1+H(\gamma \wedge 0)}_{[s,t]} L_{m,\infty}} \leq  C_0,
\end{align}
any $t \in [t_0,T]$, we have
\begin{align}\label{eq:stab-lip}
\begin{split}
\| X_t-Y_t \|_{L_{2m}} & \leq  C |x-y|  \exp\{-\kappa_1(t-t_0)+ \mathbf{M}\|b\|_{\mathcal{C}^\gamma}  (t-t_0)\} . 
\end{split}
\end{align}
\end{prop}
\begin{proof}
If \eqref{eq:stab-lip} holds for any $b \in \mathcal{C}^\infty$, then one can take $(b^n)_{n \in \N}$ be a sequence of smooth functions converging to $b$ in $\mathcal{C}^{\gamma}$ that satisfies $\sup_{n \in \N} \|b^n\|_{\mathcal{C}^\gamma} \leq \Xi$, and get that \eqref{eq:stab-lip} holds for $X^n,Y^n$, the solutions of \eqref{eq:sde} with drift $b^n$ and initial conditions $x,y$ respectively. Then using Proposition \ref{prop:stab-drift} with $g=b^n$, we get that $X_t^n$ and $Y_t^n$ converge to $X_t$ and $Y_t$ in $L_m$ as $n \rightarrow \infty$. As a result, taking the limit, we would conclude that \eqref{eq:stab-lip} also holds for $X$ and $Y$. So without any loss of generality, we assume that $b$ is a smooth function in $\mathcal{C}^\infty$.

Let $t >t_0$, using \eqref{eq:Fdissipative}, we write
\begin{align*}
\frac{d}{dt}| X_t-Y_t|^2 & = 2 \langle X_t-Y_t, F(X_t)-F(Y_t) + b(X_t)-b(Y_t) \rangle \\ 
& \leq -2 \kappa_1 | X_t-Y_t|^2 + 2 \langle X_t-Y_t , b(X_t)-b(Y_t) \rangle .
\end{align*}
Using a Gr\"onwall-type argument, we get
\begin{align*}
| X_t-Y_t |^2 & \leq e^{-2 \kappa_1( t-t_0)} |x-y|^2 +2 \int_{t_0}^t e^{- 2 \kappa_1 (t-r)} \langle X_r-Y_r, b(X_r) - b(Y_r) \rangle dr .
\end{align*}
Let $\delta >0$, $N = \lfloor t-t_0 \rfloor / \delta$ and $\Pi = \{ t_{k} \}_{k=0}^N$ be a partition such that $t_{k+1}-t_k = \delta$ for $k=0,\cdots, N-1$. Define $t_{N+1}=t$, we have for any $\tilde t \in [t_{N},t_{N+1}]$
\begin{align*}
| X_{\tilde t}-Y_{\tilde t} |^2 & \leq e^{- 2 \kappa_1 (\tilde t-t_0)} |x-y|^2 \\ & \quad + 2 \sum_{k=0}^{N-1} e^{- 2 \kappa_1 (\tilde t-t_{k+1})} \int_{t_k}^{t_{k+1}} e^{- 2 \kappa_1 (t_{k+1}-r)} \langle X_r-Y_r , b(X_r)-b(Y_r) \rangle dr \\ & \quad + \int_{t_N}^{\tilde t} e^{- 2 \kappa_1 (\tilde t-r)} \langle X_r-Y_r , b(X_r)-b(Y_r) \rangle dr
\end{align*}
By applying Proposition \ref{prop:reg3} with $\psi=X-\widetilde{B}^{t_0}$, $\phi=Y-\widetilde{B}^{t_0}$ and using \eqref{eq:inf}, bounding $\sup_{[t_N,\tilde t]}$ by $\sup_{[t_N,t_{N+1}]}$, and $\tilde t-t_N$ by $t_{N+1}-t_N$ we get
\begin{align*}
\begin{split}
\| X_{\tilde t}-Y_{\tilde t} \|_{L_{2m}}^2 &  \leq e^{- 2 \kappa_1 (\tilde t-t_0)}  |x-y|^2  \\ & \quad +C \sum_{k=0}^{N} e^{- 2 \kappa_1 (\tilde t-t_{k+1})}   \| b \|_{\mathcal{C}^\gamma}  (1+C_0) \sup_{r \in [t_{k},t_{k+1}]} \| X_r-Y_r \|_{L_{2m}}^2 (t_{k+1}-t_k)^{1+H(\gamma-1)} \\ & \quad + C \sum_{k=0}^{N} e^{- 2 \kappa_1 (\tilde t-t_{k+1})}   \| b \|_{\mathcal{C}^\gamma}    [X-Y]_{\mathcal{C}^{1/2}_{[t_k,t_{k+1}]} L_{2m}}^2 (t_{k+1}-t_k)^{3/2+H(\gamma-1)} .
\end{split}
\end{align*}
Using Lemma \ref{prop:regZ-W-lip} and \eqref{eq:inf}, we get
\begin{align}\label{eq:X-Y-progress}
\| X_{\tilde t}-Y_{\tilde t} \|_{L_{2m}}^2 
& \leq e^{-2 \kappa_1 (\tilde t-t_0)}  |x-y|^2   +C(1+C_0)^2 \| b \|_{\mathcal{C}^\gamma}  \sum_{k=0}^{N} e^{- 2 \kappa_1 (\tilde t-t_{k+1})}   \sup_{r \in [t_{k},t_{k+1}]} \| X_r-Y_r \|_{L_{2m}}^2 \delta^{1+H(\gamma-1)} .
\end{align}
Denote by $g$ the function $g(u)=e^{2 \kappa_1 (u-t_0)} \|X_u-Y_u\|_{L_{2m}}^2$, then we have
\begin{align*}
g(\tilde t) & \leq g(t_0) +  C \delta^{1+H(\gamma-1)} \| b \|_{\mathcal{C}^\gamma} \sum_{k=0}^{N}  \sup_{r \in [t_{k},t_{k+1}]} g(r) \\
& = g(t_0)+C \delta^{1+H(\gamma-1)} \| b \|_{\mathcal{C}^\gamma}  \sup_{r \in [t_{N},t_{N+1}]} g(r) +  C \| b \|_{\mathcal{C}^\gamma} \sum_{k=0}^{N-1}  \sup_{r \in [t_{k},t_{k+1}]} g(r).
\end{align*}
Taking the supremum over $\tilde t $ in $[t_N,t_{N+1}]$, we get
\begin{align*}
 \sup_{r \in [t_{N},t_{N+1}]} g(r) 
& \leq g(t_0)+C \delta^{1+H(\gamma-1)} \Xi  \sup_{r \in [t_{N},t_{N+1}]} g(r) +  C \| b \|_{\mathcal{C}^\gamma} \sum_{k=0}^{N-1}  \sup_{r \in [t_{k},t_{k+1}]} g(r).
\end{align*}
Choosing $\delta = (1/(2 C \Xi))^{\frac{1}{1+H(\gamma-1)}}$, we have
\begin{align*}
\sup_{t \in [t_{N},t_{N+1}]} g(t) \leq 2 g(t_0) +   2 C \| b \|_{\mathcal{C}^\gamma} \sum_{k=0}^{N-1}  \sup_{r \in [t_{k},t_{k+1}]} g(r) \\
\end{align*}
A discrete Gr\"onwall argument shows that
\begin{align*}
g(t) \leq 2 g(t_0) \exp\Big\{2C \| b \|_{\mathcal{C}^\gamma}  (N+1) \Big\}  \leq 2 g(t_0) \exp\Big\{2C \delta^{-1} \| b \|_{\mathcal{C}^\gamma}  (t-t_0) \Big\} ,
\end{align*}
which concludes the proof. Note that the constant in front $g(t_0)$ can be chosen arbitrarily close to $1$, and this changes the constant in the exponential.
\end{proof}

\subsection{Weak existence: Proof of Theorem \ref{thmmain-existence}}\label{sec:weakexistence}
The goal of this section is to prove Theorem \ref{thmmain-existence}. The proof is based on performing a tightness argument based on the a priori estimates of Section \ref{sec:apriori}.
Assume that $F$ satisfies Assumption \ref{assumpF}. Let $\gamma \in (1/2-1/(2H),1)$ and let $b$ satisfy Assumption \ref{assumpb}.
Let $x \in \R^d$, $T \in \R$, $t_0 <T$, and $\left(b^n\right)_{n \in \mathbb{N}}$ be a sequence of bounded continuous functions converging to $b$ in $\mathcal{C}^{\gamma}$. Assume without any loss of generality that $\sup_{n \in \N} \|b^n \|_{\mathcal{C}^\gamma} \leq \Xi$. 
\paragraph{Tightness.}
Let $X^n$ be a solution of \eqref{eq:sde} with initial condition $x$ and drift $b^n$ defined on some filtered probability space and with $W$ a Wiener process with respect to its underlying filtration. For any $m \in [2,\infty)$, by Proposition \ref{prop:regZ} and Proposition \ref{prop:unif-moments}, there exists a constant $C:=C(m,\gamma,H,x,\Xi,T-t_0)$ such that for any $(s,t) \in \Delta_{[t_0,T]}$ with $t-s \leq 1$
\begin{align*}
\| X^n_t-B_t-(X^n_s-B_s) \|_{L_m} \leq C (t-s)^{1+H(\gamma \wedge 0)} .
\end{align*}
Hence using classical tightness arguments (e.g. \cite[Theorem 7.3]{Billingsley}), one can deduce that the sequence $(X^n,W)_{n \in \N}$ is tight in $\mathcal{C}([t_0,T],\R^d) \times \mathcal{C}(\R,\R^d)$.
%where $\mathcal{C}([t_0,T])$ is the space of real continuous functions equipped with the topology of uniform convergence over compact sets and $\mathcal{C}(\R)$ is the space of real continuous functions also equipped with the metric
%\begin{align*}
%    d(f,g) =\sum 2^{-k} \left( 1 \wedge \sup_{r \in [-2^k,2^k]} | f(r)-g(r)| \right).
%\end{align*}
Since this space is Polish (recall the metric \eqref{eq:metric}), by the Prokhorov theorem, there exists a subsequence $(n_k)_{k \in \N}$ such that $(X^{n_k},W)$ converges weakly in the space $\mathcal{C}([t_0,T],\R^d) \times \mathcal{C}(\R,\R^d)$. 

\paragraph{Conclusion.} By passing to the subsequence, we can assume that $(X^{n},W)$ converges weakly in the space $\mathcal{C}([t_0,T],\R^d) \times \mathcal{C}(\R,\R^d)$. Since it is a Polish space, we can apply the Skorokhod representation theorem and deduce that there exists a sequence $(\widehat Y^n,\widehat W^n)$ defined on a common probability space $(\widehat \Omega, \widehat{\mathcal{F}}, \widehat{\mathbb{P}})$ and a random element $(\widehat{X},\widehat{W})$ such that $(\widehat Y^n, \widehat W^n)=(X^n,W)$ in law and $(\widehat Y^n, \widehat W^n)$ converges to $(\widehat{X},\widehat{W})$ almost surely. In particular, using the continuity of the drift $b^n$ and $F$, one can show that $\widehat Y^n$ is a solution to \eqref{eq:sde} with respect to $(\widehat \Omega,\widehat{\mathcal{F}}^{\widehat W^n}, \widehat{\mathbb{P}})$ and $\widehat W^n$ instead of $W$. 

Following the same proof as in \cite[Proposition 3.4]{athreya2020well} and using that $(\widehat Y^n, \widehat W^n)$ converges to $(\widehat{X},\widehat{W})$ almost surely, and satisfies \eqref{eq:thm-existence-2} and \eqref{eq:thm-existence-3} with an fBm $\widehat B^n$ defined via $\widehat W^n$, one can show that \eqref{approximation2} holds where $K$ is the process defined by 
\begin{align*}
    K_t = \widehat{X}_t-x-\int_{t_0}^t F(\widehat{X}_r)dr -\widehat{B}_t + \widehat{B}_{t_0}, \, t \ge t_0
\end{align*}
with $\widehat{B}$ defined via $\widehat{W}$,
and that $(\widehat{X},\widehat{B})$ also satisfies the regularity estimates \eqref{eq:thm-existence-2} and \eqref{eq:thm-existence-3}.

For $t \in [t_0,T]$, define now $\widetilde{\mathcal{F}_t}:=\sigma((\widehat{X}_r)_{r \in [t_0,t]}, (\widehat{W}_r)_{r \leq t})$ and denote by $\widetilde{\mathcal{F}},\widetilde{\mathbb{F}}$ the corresponding sigma algebra and filtration, then clearly $\widehat{X}_t$ is $\widetilde{\mathcal{F}}_t$-measurable. Hence $(\widehat{X}_t)_{t \in [t_0,T]}$ is a solution with respect to $(\widehat \Omega,\widetilde{\mathcal{F}}, \widetilde{\mathbb{F}}, \widehat{\mathbb{P}})$ and $\widehat{W}$. 

\subsection{Existence and uniqueness of solutions: Proof of Theorem \ref{thmmain-uniqueness}}\label{sec:strongexistence}
The goal of this section is to prove Theorem \ref{thmmain-uniqueness}. The proof is based on mollifying the drift, applying the regularisation properties and passing to the limit. Let $(\Omega,\mathcal{F},\mathbb{F},\mathbb{P})$ be a filtered probability space and $(W_t)_{t \in \R}$ be a two-sided $\mathbb{F}$-Wiener process. Solutions in this section will always be understood with respect to this probability space and $W$.
Assume that $F$ satisfies Assumption \ref{Flip}. Let $\gamma \in (1-1/(2H),1)$ and $b$ satisfy Assumption \ref{assumpb}. Let $x \in \R^d$, $T \in \R$, $t_0 <T$, and $\left(b^n\right)_{n \in \mathbb{N}}$ be a sequence of bounded continuous functions converging to $b$ in $\mathcal{C}^{\gamma}$. Assume without any loss of generality that $\sup_{n \in \N} \|b^n \|_{\mathcal{C}^\gamma} \leq \Xi$. 

For any $n \in \N$, let $X^n$ be the solution to \eqref{eq:sde} with initial condition $x$, drift $b^n$ and fBm $B$ defined via $W$.

\paragraph{Existence.} Let $m \in [2,\infty)$ and $n,k \in \N$, by Proposition \ref{prop:regZ} and Proposition \ref{prop:regZ-infty}, we have that $X^n,X^k$ belong to the class $\mathcal{V}(m)$ defined in \eqref{eq:classVm} and $X^n$ satisfies \eqref{eq:inf}. Thus, we can apply Proposition \ref{prop:stab-drift} to get that there exists a constant $C:=C(m,\gamma,H, \Xi)$ such that for any $t_0 \leq s <t$ with $t-s \leq 1$, we have
\begin{align}\label{eq:Xn-Xk-C1/2}
[ X^n-X^k ]_{\mathcal{C}^{1/2}_{[s,t]} L_m} & \leq C  \| b^n-b^k \|_{\mathcal{C}^{\gamma}} .
\end{align}
Hence, by Kolmogorov's continuity theorem (choosing $m$ sufficiently large), there exists $\varepsilon \in (0,1/2)$ and a constant $C:=C(\varepsilon,m,\gamma,H,T-t_0, \Xi)$ such that
\begin{align*}
\Big\|  \sup_{(s,t) \in \Delta_{[t_0,T]}} \frac{|X^n_t-X^k_t-(X^n_s-X^k_s)|}{(t-s)^{\frac{1}{2}-\varepsilon}} \Big\|_{L_m} & \leq C \| b^n-b^k \|_{\mathcal{C}^{\gamma}} .
\end{align*}
Since $X^n_{t_0}=X^k_{t_0}=x$, it follows that $\| \sup_{t \in [t_0,T]} |X^n_t-X^k_t| \|_{L_m}$ converges to $0$ as $n,k \rightarrow \infty$. Therefore, $(X^n)_{n \in \N}$ is a Cauchy sequence in $L_m(\Omega; \mathcal{C}([t_0,T],\R^d))$ and thus converges to a continuous process $X$. Since for any $n$, $X^n$ is adapted to $\mathbb{F}$, it follows that $X$ is also adapted to $\mathbb{F}$. \\

We will show that $X$ is a solution to \eqref{eq:sde}. First we check that it satisfies the regularity estimates \eqref{eq:thm-existence-1}, \eqref{eq:thm-existence-2} and \eqref{eq:thm-existence-3}, then show that it satisfies \eqref{solution1} from Definition \ref{defsol-SDE}.
By Propositions \ref{prop:regZ}, \ref{prop:unif-moments} and \ref{prop:regZ-infty}, we have that for any $n \in \N$, $X^n$ satisfies \eqref{eq:thm-existence-1}, \eqref{eq:thm-existence-2} and \eqref{eq:thm-existence-3}. Since the constants do not depend on $n$, by passing to the limit, we get that $X$ also satisfies the same estimates. 

Next, we check that $X$ satisfies \eqref{solution1} of Definition \ref{defsol-SDE} with a process $K$ that satisfies \eqref{approximation2}. For any $n \in \N$, define $K^n = \int_{t_0}^{\cdot} b^n(X_r) dr$.

Using Proposition \ref{prop:reg1} (with $q=m$, $\psi=X-\widetilde{B}$, $\gamma \equiv \gamma-1$, $\tau=1+H(\gamma \wedge 0)$ and $f=b^n-b^j$), there exists a constant $C:=C(m,\gamma,H,T-t_0)$ have that for any $n,j \in \N$ and $t_0 <s <t$
\begin{align*}
\| K^n_t-K^n_s - (K^j_t-K^j_s) \|_{L_m} & \leq C \| b^n-b^j\|_{\mathcal{C}^{\gamma-1}} \Big( 1+ \llbracket X-\widetilde{B} \rrbracket_{\mathcal{C}^{1+H(\gamma \wedge 0)}_{[s,t]} L_{m}} \Big) (t-s)^{1+H(\gamma-1)} .
\end{align*}
Using that $\llbracket X-\widetilde{B} \rrbracket_{\mathcal{C}^{1+H(\gamma \wedge 0)}_{[s,t]} L_{m}} \leq \llbracket X-\widetilde{B} \rrbracket_{\mathcal{C}^{1+H(\gamma \wedge 0)}_{[s,t]} L_{m,\infty}}= \llbracket X-B\rrbracket_{\mathcal{C}^{1+H(\gamma \wedge 0)}_{[s,t]} L_{m,\infty}}$ and \eqref{eq:thm-existence-3}, we get that
\begin{align*}
\| K^n_t-K^n_s - (K^j_t-K^j_s) \|_{L_m} \leq C \|b^n-b^j\|_{\mathcal{C}^{\gamma-1}} (t-s)^{1+H(\gamma-1)}.
\end{align*}
Consequently, we deduce (again by Kolmogorov's continuity theorem) that $(K^n)_{n \in \N}$ is a Cauchy sequence in $L_m(\Omega,\mathcal{C}([t_0,T],\R^d))$ and therefore converges to some continuous process $K$. In particular, $K$ satisfies \eqref{approximation2} of Definition \ref{defsol-SDE}.

To conclude that $X$ is a solution to \eqref{eq:sde}, we show that $X$ satisfies \eqref{solution1}. Recall that $(X^n)_{n \in \N}$ satisfies \eqref{solution1} with $K$ replaced with $\int b^n(X^n)$. Using that $F$ is Lipschitz, we have for any $t \in [t_0,T]$,
\begin{align*}
    & \| X_t-x-\int_{t_0}^t F(X_r)dr- K_t -(B_t-B_{t_0}) \|_{L_m} \\ & = \|X_t-X_t^n + \int_{t_0}^tF(X^n_r)-F(X_r)dr + \int_{t_0}^t b^n(X^n_r) -b^n(X_r) dr + K^n_t-K_t \|_{L_m} \\
    & \leq \|X_t-X_t^n\|_{L_m} + C \sup_{r \in [t_0,T]} \|X_r-X^n_r\|_{L_m} (t-t_0) + \|\int_{t_0}^t b^n(X^n_r)-b^n(X_r) dr \|_{L_m} + \| K^n_t-K_t\|_{L_m} .
\end{align*}
Using Proposition \ref{prop:reg2} with $f=b^n$, $\tau=1/2$, and recalling that $X$ satisfies \eqref{eq:thm-existence-3} and $\| b^n\|_{\mathcal{C}^\gamma} \leq \Xi$, we get that there exists a constant $C:=C(m,\gamma,H,t_0,T,\Xi)$ such that
\begin{align*}
    & \| X_t-x-\int_{t_0}^t F(X_r)dr- K_t-(B_t-B_{t_0}) \|_{L_m}
    \\ & \leq \|X_t-X_t^n\|_{L_m} +C \sup_{r \in [t_0,T]} \|X_r-X^n_r\|_{L_m}  + C [X^n-X]_{\mathcal{C}^{1/2}_{[t_0,T] L_m}} + \| K^n_t-K_t\|_{L_m} .
\end{align*}
By passing to the limit as $k \rightarrow \infty$ in \eqref{eq:Xn-Xk-C1/2}, we get that $[X^n-X]_{\mathcal{C}^{1/2}_{[t_0,T] L_m}} \leq C \| b^n-b\|_{\mathcal{C}^\gamma}$. Hence, passing to the limit as $n \rightarrow \infty$ and using that $X^n,K^n$ converge to $X,K$ in $L_m(\Omega,\mathcal{C}([t_0,T],\R^d))$,  we get that for any $t \in [t_0,T]$, almost surely, $$X_t=x+\int_0^t F(X_r)dr + K_t + B_t-B_{t_0}.$$
Since the processes in the equation above have H\"older continuous trajectories, it follows that $X$ satisfies \eqref{solution1}, and is therefore a solution to \eqref{eq:sde}.

\paragraph{Uniqueness.} Let $Y$ be another solution to \eqref{eq:sde} in the class $\mathcal{V}$ (defined in \eqref{eq:class-uniqueness}). Recalling that $X$ satisfies \eqref{eq:thm-existence-3}, we can apply Proposition \ref{prop:stab-lip} (since $F$ is Lipschitz, it satisfies \eqref{eq:Fdissipative} for some $\kappa_1<0$ and 
$\kappa_2=0$), and take $x=y$, $m=2$ to get that $X$ and $Y$ are indistinguishable.

\paragraph{Proof of \eqref{eq:exp-decay}.} Since the unique solution of \eqref{eq:sde} satisfies \eqref{eq:inf}, we can apply Proposition \ref{prop:stab-lip} with $C_0=C_0(m,\gamma,H,L_F,\Xi)$ to obtain \eqref{eq:exp-decay}.

 \section{Markovian structure and invariant measures}\label{sec:longtime}
The goal of this section is to prove the existence and uniqueness of an invariant measure.
We fix a probability space $(\Omega, \mathcal{F}, \mathbb{F}, \mathbb{P})$ and an $\mathbb{F}$-Wiener process $W$. 
Let $\mathcal{C}_0((-\infty,0],\R^d)$ (respectively $\mathcal{C}^\infty_0((-\infty,0],\R^d)$) denote the space of $\R^d$-valued continuous (respectively smooth) functions over $(-\infty,0]$ that vanish at $0$. Let  $\mathbf{W}$ be the left-sided Wiener measure, i.e. the measure induced by $W|_{(-\infty,0]}$ on $\mathcal{C}_0((-\infty,0],\R^d)$.
Let $\mathcal{H}_H$ be the closure of $\mathcal{C}^\infty_0((-\infty,0],\Rd)$ with respect to the norm
\begin{align}
    \| w \|_{\mathcal{H}_H} := \sup_{s,t \in (-\infty,0]} \frac{|w(t)-w(s)|}{|t-s|^{(1-H)/2} \big( 1+ |t|+ |s| \big)^{1/2}}.
\end{align}
Then, $\mathcal{H}_H$ is Polish and $\mathbf{W}$ is supported on $\mathcal{H}_H$ (\cite[Lemma 3.10 and Lemma 3.8]{Hairer}). 
\paragraph{Framework and notation.}
The first enhancement is to allow dynamics of \eqref{eq:sde} to start from a given memory.
To this end, it is necessary to replace the driving signal up to time $t_0$, i.e. $\bar{B}^{t_0}$, by a generic $\mathcal{F}_{t_0}$-measurable continuous path $V$ and consider the corresponding SDE to \eqref{eq:sde} starting from $(X_0,V)$ at time $t_0$. More precisely, we consider the evolution $t \mapsto X^{(X_0,V)}(t_0,t)$ given by the SDE
\begin{align}\label{eq:sde-eta}
\begin{split}
d X^{(X_0,V)}(t_0,t) & = F(X^{(X_0,V)}(t_0,t) )dt + b(X^{(X_0,V)}(t_0,t)) dt + d V_{t} + d \widetilde{B}^{t_0}_{t} , \quad t \ge t_0, \\
X^{(X_0,V)}(t_0,t_0) &  = X_0.
\end{split}
\end{align}
The initial input at time $t_0$ consists of $X_0$ and $V$ which are $\mathcal{F}_{t_0}$-random variables in $\Rd$ and $\mathcal{C}_0([0,\infty),\Rd)$ respectively.
The innovation $\widetilde{B}^{t_0}$ plays the role of driving signal.
Given $t_0,X_0,V$,
one can define a solution with respect to $(\Omega, \mathcal{F}, \mathbb{F}, \mathbb{P})$ and $W$ in a similar fashion to Definition \ref{defsol-SDE}.
Earlier results for \eqref{eq:sde} rely solely on the regularisation properties of the innovation $\widetilde{B}^{t_0}$, and therefore, can be carried forward to equation \eqref{eq:sde-eta}, see Theorem \ref{thm:Phi-results} below.
In particular, when $x\in\Rd$ and $V=\bar B^{t_0}$, \eqref{eq:sde-eta} becomes \eqref{eq:sde}, and hence, by uniqueness, $(X^{(x,\bar B^{t_0})}(t_0,t))_{t \ge 0}$ is the solution to \eqref{eq:sde} starting at $x$. It will also be useful to consider the evolution from a deterministic initial input $(x,v)$. In Theorem \ref{cor:markov}, we show that there exists a jointly measurable map
\begin{align}\label{eq:Xmes}
    (t_0,t,x,v) \mapsto X(t_0,t,x,v) ,
\end{align}
such that for any $t \ge t_0$, almost surely in $\omega$, $X(t_0,t,x,v) = X^{(x,v)}(t_0,t)$. We also prove that for any $\mathcal{F}_{t_0}$-random variable $(X_0,V)$ and $t \ge t_0$, we have $X(t_0,t,X_0(\omega),V( \omega))(\omega)=X^{(X_0,V)}(t_0,t)(\omega)$ almost surely in $\omega$.
 
The map \eqref{eq:Xmes} will be important in defining a Markov evolution. In fact, a key point of \cite{Hairer} is that while the evolution \eqref{eq:sde} is not Markovian, the joint evolution of \eqref{eq:Xmes} and the increments of the noise is.
The evolution of the noise component takes place in the state space $\mathcal{H}_H \subset \mathcal{C}_0((-\infty,0],\R^d)$.
Then one defines for any $w \in \mathcal{H}_H$ and $r \leq 0$, the process 
\begin{align}\label{def:Z}
    Z(t_0,t,w)(r) := \bar{Z}(t_0,t,w)(r)-\bar{Z}(t_0,t,w)(0),
\end{align}
where
\begin{align*}
\bar{Z}(t_0,t,w)(r) =\;
\begin{cases}
w\bigl(r + t-t_0 \bigr),
& \text{ when } r+t \le t_0,\\[6pt]
 W_{r+t}-W_{t_0},
& \text{ when } t_0 < r+t,
\end{cases}
\end{align*}
Heuristically, this corresponds to concatenating the Wiener path $(W_r-W_{t_0})_{r\ge t_0}$ to $(w({r-t_0}))_{r\le t_0}$, and, to ensure that resulting process stays in $\mathcal{H}_H \subset \mathcal{C}_0((-\infty,0],\R^d)$, re-parametrising the time parameter to $(-\infty,0]$, then re-centring the process in space to $0$ at time $0$. If one ignores the time re-parametrisation, then  for each $t\ge t_0$,  $Z(t_0,t,w)$ corresponds to the entire trajectory up to time $t$ of the resulted centred concatenated path.
Note that if $w \equiv (W_{r+t_0}-W_{t_0})_{r \leq 0}$, then the resulting process $(Z (t_0,t,w)(r))_{r \leq 0}$ is nothing but $(W_{t+r}-W_{t})_{r\le 0}$, which has law $\mathbf{W}$.

To couple $Z(t_0,t,w)$ with \eqref{eq:Xmes}, we use the continuous map (see \cite[Lemma 3.8]{Hairer})
\begin{equation}\label{eq:operatorA} 
    \begin{aligned}
        \mathbf{A}^{t_0} : \mathcal{H}_H &\rightarrow \mathcal{C}([t_0,\infty),\Rd)
        \\ w &\mapsto  \Big( t \mapsto \alpha_H \int_{-\infty}^{t_0} \big( (t-u)^{H-1/2}-(t_0-u)^{H-1/2} \big) d w(u-t_0) \Big), 
    \end{aligned}
\end{equation}
and  consider the jointly measurable map $\Theta(t_0,t,\cdot,\cdot)$ over $\R^d \times \mathcal{H}_H$ defined by
\begin{align}\label{def:xi}
\Theta(t_0,t,x,w) = \left( X(t_0,t,x,\mathbf{A}^{t_0}(w)), Z(t_0,t,w) \right)
\end{align}
and for each $s\le t$, the corresponding operator
\begin{align}\label{def:semigroup}
(P_{s,t}f)(x,w) = \EE f(\Theta(s,t,x,w)),
\end{align}
for any bounded measurable function $f : \R^d \times \mathcal{H}_H \rightarrow \R$.

To connect the evolution $\Theta$ with \eqref{eq:sde}, \cite{Hairer} defines a generalised initial condition as a probability measure $\mu_0 \in \mathcal{P}(\R^d \times \mathcal{H}_H)$ such that its second marginal is  the left-sided Wiener measure $\mathbf{W}$.
It turns out (Theorem \ref{cor:markov}) that the evolution $\Theta$ is a Markov process with evolution family $P$ (see Definition \ref{def:semigroup}), which  becomes time-homogeneous when started from a generalised initial condition. 
Furthermore, because the second marginal of the law of $\Theta$ started from a generalised initial condition is always $\mathbf{W}$, the set of generalised initial conditions is invariant under $(P_{s,t})$. Hence, defining $P_t=P_{0,t}$, then $(P_t)_{t\ge0}$ forms a transition semigroup on the set of generalised initial conditions. 
As in \cite{Hairer}, we call $\mu$ an invariant measure (see Definition \ref{def:inv}) associated to \eqref{eq:sde-eta} if it is a generalised initial condition and it is invariant under the transition semigroup $(P_t)$, that is $P^*_{t} \mu = \mu$ for any $t \ge 0$ where $P_t^*$ is the adjoint operator.

In \cite{Hairer}, the invariant measure is built within the more general stochastic dynamical systems (SDS) framework, where one constructs the SDS from the deterministic flow corresponding to the SDE (seen as a function of time, the initial condition and the driving signal). 
When $b$ is smooth (so that \eqref{eq:sde-eta} can be treated as a random ordinary differential equation), the first component of $\Theta$ corresponds to the SDS and the second component corresponds to the stationary noise process. However, this approach is not well adapted when $b$ is singular since the solution of \eqref{eq:sde-eta} is not given as a continuous function of the driving noise. We have adopted a simpler construction, relying solely on the theory of Markov processes, which is well suited to our context.

\paragraph{Main results.}
As we have mentioned before, the well-posedness results for \eqref{eq:sde-eta} can be obtained by replacing $B$ by $\widetilde{B}^{t_0}$. This means changing the semi-norms defined in \eqref{eq:semi-norm} and \eqref{eq:semi-norm-2} -when applied to $X-B$ (where $X$ is the solution of \eqref{eq:sde})- in the following way: For any $\alpha>0$, any interval $I$, any $q \in [2,\infty)$, and any $m \in [q,\infty)$,
\begin{itemize}
    \item $\llbracket X- B \rrbracket_{\mathcal{C}^{\alpha}_{I} L_{m,q}}$ is replaced by $\llbracket X^{(X_0,V)}(t_0,\cdot) -\widetilde{B}^{t_0}-V \rrbracket_{\mathcal{C}^{\alpha}_{I} L_{m,q}} = \llbracket X^{(X_0,V)}(t_0,\cdot) -\widetilde{B}^{t_0} \rrbracket_{\mathcal{C}^{\alpha}_{I} L_{m,q}}$.
    \item $[X-B]_{\mathcal{C}^{\alpha}_{I} L_{m}}$ is replaced by $[X^{(X_0,V)}(t_0,\cdot)-\widetilde{B}^{t_0} -V]_{\mathcal{C}^{\alpha}_{I} L_{m}}$.
\end{itemize}
Moreover, since the innovation $(\widetilde{B}^{t_0}_t)_{t \ge t_0}$ is independent of $\mathcal{F}_{t_0}$, the well-posedness analysis of Section \ref{sec:well-posedness} still holds for an $\mathcal{F}_{t_0}$ random variable $(X_0,V)$. This holds for example if $X_0$ is $\mathcal{F}_{t_0}$-measurable and $V=\mathbf{A}^{t_0}(\bar W)$ where $\bar W$ is a continuous random process such that $\bar W_t$ is $\mathcal{F}_{t_0}$-measurable for any $t \leq 0$.
Well-posedness results for the SDE \eqref{eq:sde-eta} will be summarised in Theorem \ref{thm:Phi-results}, they are to be compared to Theorem \ref{thmmain-existence}, Theorem \ref{thmmain-uniqueness}, Proposition \ref{prop:stab-drift} and Proposition \ref{prop:stab-lip}.  Their proof is omitted, and the dependence of constants on $V$ is clarified. In particular, recall that the uniform-in-time bound on the moments goes through a comparison with the Ornstein-Uhlenbeck process, now it will go through a comparison with the process
\begin{align*}
    dU^{(X_0,V)}(t_0,t)  =  -U^{(X_0,V)}(t_0,t) dt + d \eta_t + d \widetilde{B}^{t_0}_t, \quad 
    U^{(X_0,V)}(t_0,t_0)  = X_0 .
\end{align*}
In particular, one can show (as in the proof of Proposition \ref{prop:unif-moments}) that the $m$-th moment of $X^{(X_0,V)}(t_0,t)-U^{(X-0,V)}(t_0,t)$ is bounded as follows:
\begin{align*}
 \|    X^{(X_0,V)}(t_0,t)-U^{(X_0,V)}(t_0,t)\|_{L_m} \leq C (1+C_0(\kappa,T-t_0) \sup_{r \in [t_0,T]} \|U^{(X_0,V)}(t_0,r)\|_{L_m})
\end{align*}
where $C_0$ is given by \eqref{eq:defC0}. Notice that $U^{(X_0,V)}(t_0,\cdot)$ is a Gaussian process with bounded variance, and its mean $m^{(X_0,V)}(t_0,t)$ is given by 
\begin{align*}
    m^{(X_0,V)}(t_0,t) := e^{-(t-t_0)} \EE (X_0)+ \EE (V_t)-e^{-(t-t_0)} \EE(V_{t_0}) - \int_{t_0}^{t} e^{-(t-u)} \EE(V_{u}) du .
\end{align*}
Hence, there exists a constant $C=C(m,H,x)$ such that for any $t \ge t_0$,
\begin{align*}
    \| U^{(X_0,V)}(t_0,t) \|_{L_m} \leq C (|m^{(X_0,V)}(t_0,t)|+1) . 
\end{align*}
Therefore, the bound on the moments of $X^{(X_0,V)}$ in the following theorem will depend on $V$ through $m^{(X_0,V)}(t_0,t)$.

\begin{theorem}\label{thm:Phi-results}
Let $m \in [2,\infty)$ and $\gamma \in (1-1/(2H),1)$. Let $b$ satisfy Assumption \ref{assumpb}, $F$ satisfy Assumptions \ref{assumpF} and \ref{Flip}. Let $T \in \R$, $t_0<T$, $X_0: \Omega \rightarrow \R^d$ be an $\mathcal{F}_{t_0}$-measurable random variable, and $V: \Omega \rightarrow \mathcal{C}([t_0,+\infty), \R^d)$ be a random variable such that $V_t$ is $\mathcal{F}_{t_0}$-measurable for any $t \ge t_0$. Then there exists a unique solution $(X^{(X_0,V)}(t_0,t))_{t \in [t_0,T]}$ to \eqref{eq:sde-eta} in the class
\begin{align}\label{eq:class-psi}
\mathcal{V}^{V}:= \{ Y: \,  [ Y-\widetilde{B}^{t_0}-V ]_{\mathcal{C}^{1/2}_{[t_0,T]} L_2} < \infty \}.
\end{align}
Moreover, the following statements hold.
\begin{enumerate}[label=(\roman*)]
\item There exist $C_1:=C_1(m,\gamma,H,\kappa,\Xi)$, $C_2:=C_2(m,\gamma,H,\kappa,\Xi)$ and $C_3:=C_3(m,\gamma,H,L_F,\Xi)$ such that for any $t_0 \leq s <t$ with $t-s \leq 1$, we have
 \begin{align}\label{eq:unif-Phi}
  \| X^{(X_0,V)}(t_0,t) \|_{L_m} & \leq C_1 (1+\sup_{r \in [t_0,T]} |m^{(X_0,V)}(t_0,r)|) \Big(1 + (T-t_0) \wedge \frac{1-e^{-\kappa_1(T-t_0)} }{\kappa_1}    \Big),
  \end{align}
  \begin{align*}
[X^{(X_0,V)}(t_0,\cdot)-\widetilde{B}^{t_0}-V ]_{\mathcal{C}^{1+H(\gamma \wedge 0)}_{[s,t]} L_{m}} & \leq C_2 (1+ \sup_{r \in [s,t]} \| X^{(X_0,V)}(t_0,r) \|_{L_m})  ,
\end{align*}
and
\begin{align}\label{eq:inf-Phi}
 \llbracket X^{(X_0,V)}(t_0,\cdot)-\widetilde{B}^{t_0}  \rrbracket_{\mathcal{C}^{1+H(\gamma \wedge 0)}_{[s,t]} L_{m,\infty}} \leq C.  
 \end{align}
 \item Let $g$ satisfy Assumption \ref{assumpb}, and denote by $X_g^{(X_0,V)}$ the unique solution to \eqref{eq:sde-eta} in the class $\mathcal{V}^V$, with initial input $(X_0,V)$ and singular drift $g$, then there exists a constant $C:=C(m,\gamma,H,L_F,\Xi,T-t_0)$ such that for all $t \in [t_0,T]$
\begin{align}\label{eq:cont-drift}
\| X^{(X_0,V)}(t_0,t) -X_g^{(X_0,V)}(t_0,t)  \|_{L_m} \leq C \|b-g\|_{\mathcal{C}^{\gamma-1}} .
\end{align}
 \item If $\kappa_2=0$, then there exists a universal constant $C$ and a constant $\mathbf{M}=\mathbf{M}(m,\gamma,H,L_F,\Xi)$ such that for any $\mathcal{F}_{t_0}$-measurable random variable $Y_0$ and any $t \in [t_0,T]$, we have
\begin{align}\label{eq:decay-Phi}
\begin{split}
 &  \| X^{(X_0,V)}(t_0,t)-X^{(Y_0,V)}(t_0,t) \|_{L_{m}} \\ & \leq \|X_0-Y_0\|_{L_m} \exp \{(-\kappa_1+\mathbf{M} \|b\|_{\mathcal{C}^\gamma} (1+\|b\|_{\mathcal{C}^\gamma})^{\frac{4}{1+H(\gamma-1)}}) (t-t_0) \} . 
\end{split}
\end{align}
\end{enumerate}
\end{theorem}
The well-posedness assumptions of Theorem \ref{thm:Phi-results} will always be assumed in the following results. From now on, for any $t_0 \in \R$, for any $\mathcal{F}_{t_0}$-measurable random variable $(X_0,V): \Omega \rightarrow \R^d \times \mathcal{C}([t_0,+\infty),\R^d)$, $X^{(X_0,V)}(t_0,\cdot)$ will denote the unique solution to \eqref{eq:sde-eta} in the class $\mathcal{V}^V$.
The following theorem proves the existence of the map $(x,v) \mapsto X(t_0,t,x,v)$, hinted at in \eqref{eq:Xmes}, for any $t \ge t_0$ and states that $\Theta$ defined in \eqref{def:xi} is a Markov process and justifies defining an invariant measure as a generalised initial condition that satisfies $P^*_t \mu = \mu$ for all $t \ge 0$. This is a classical result when $b$ is smooth. Here the challenge lies in proving the result for singular $b$. First, we recall the definition of evolution families and time non-homogeneous Markov processes.

\begin{definition}
Let $t_0 \in \R$ and $E$ be a Polish space, and denote by $\mathcal{B}_b(E)$ the space of bounded measurable functions $f:E \rightarrow \R$.

We say that $P:=(P_{s,t})_{t_0 \leq s \leq t}$ is a evolution family on $E$ if for any $t_0 \leq s \leq u \leq t$ $P_{s,t}$ maps $\mathcal{B}_b(E)$ to $\mathcal{B}_b(E)$, $P_{s,s}$ is the identity operator and $P_{s,u}P_{u,t}=P_{s,t}$. We say that $P:=(P_{s,t})_{t_0 \leq s \leq t}$ is a Feller evolution family if it is an evolution family and if for any $t_0 \leq s \leq t$, $P_{s,t}$ maps $\mathcal{C}(E)$ to $\mathcal{C}(E)$.

Given an evolution family $P$, we say that $(Y_t)_{t \ge t_0}$ is a Markov process with evolution family $P$ if for any $t_0 \leq s \leq t$ and any $f \in \mathcal{B}_b(E)$, we have
\begin{align*}
    \mathbb{E}^s f(Y_t) = P_{s,t}f(Y_s) \text{ a.s}.
\end{align*}
\end{definition}

\begin{theorem}\label{cor:markov}
Let $\gamma \in (1-1/(2H),1)$. Let $b$ satisfy Assumption \ref{assumpb} and $F$ satisfy Assumptions \ref{assumpF} and \ref{Flip}. Then the following statements hold.
\begin{enumerate}[label=(\roman*)]
\item For any $t_0 \in \R$ and any $t \ge t_0$, there exists a jointly-measurable map $(x,v) \mapsto X(t_0,t,x,v)$ such that for all $(x,v) \in \R^d \times \mathcal{C}([t_0,\infty),\R^d)$,
\begin{align*}
 \mathbb{P}(X^{(x,v)}(t_0,t)=X(t_0,t,x,v))=1 . 
\end{align*}
\item For any $t_0 \in \R$ and any $(x,w) \in \R^d \times \mathcal{H}_H$, define the evolution $\Theta:=(\Theta(t_0,t,x,w))_{t \ge t_0}$ as in \eqref{def:xi} and the family of operators $P:=(P_{s,t})_{t_0 \leq s \leq t}$ as in \eqref{def:semigroup}. Then, $P$ is a Feller evolution family on $\R^d \times \mathcal{H}_H$, and $\Theta$ is a Markov process with evolution family $P$.
\item For any $t_0 \in \R$, any $t_0 \leq s \leq t$ and any generalised initial condition $\mu_0$, we have that $P^*_{s,t}\mu_0=P^*_{0,t-s}\mu_0$, and $(P_r^*\mu_0)_{r\ge0} := (P^*_{0,r}\mu_0)_{r\ge0}$ forms a semigroup, where for any $r\ge 0$, $P^*_{0,r}$ denotes the adjoint of the operator $P_{0,r}$.
\end{enumerate}
\end{theorem}

The proof of Theorem \ref{cor:markov} is established in Section \ref{subsec:link-SDEs}. We now provide a rigorous definition of an invariant measure.
\begin{definition}\label{def:inv}
Under the assumptions of Theorem \ref{cor:markov}, we say that $\mu$ is an invariant measure associated to \eqref{eq:sde-eta} if $\mu$ is a generalised initial condition and satisfies $P^*_t \mu=\mu$ for any $t \ge t_0$.
\end{definition}
The final theorem deals with the existence and uniqueness of invariant measures when $\kappa_1>0$.

\begin{theorem}\label{thm inv}
Let $\gamma \in (1-1/(2H),1)$ and let $b$ satisfy Assumption \ref{assumpb}. Let $F$ satisfy Assumptions \ref{assumpF} and \ref{Flip} . Then the following statements hold.
\begin{enumerate}[label=(\roman*)]
\item If $\kappa_1 > 0$ in \eqref{eq:Fdissipative}, then there exists an invariant measure associated to \eqref{eq:sde-eta} that has finite moments.
\item Let $m \in [2,\infty)$ and recall the constant $\beta(m,\|b\|_{\mathcal{C}^\gamma})$ defined in \eqref{eq:defbeta}. If $\kappa_1 >0$, $\kappa_2=0$ and the $\mathcal C^{\gamma}$-norm of $b$ is sufficiently small so that $\beta(2,\|b \|_{\mathcal{C}^{\gamma}}) <0$, then there exists a unique invariant measure $\mu$ associated to the SDE \eqref{eq:sde-eta}. Moreover, there exists a constant $C:=C(m,\gamma,H,\kappa,\Xi)$ such that for any $t_0 \in \R$, any $t \ge t_0$ and any $\mathcal{F}_{t_0}$-measurable random variable $(X_0,\bar W)$ whose law is a generalised initial condition, we have
\begin{align}\label{eq:conv-inv}
        \mathcal{W}_m(\mathcal{L}(\Theta^{(X_0,\bar W)}(t_0,t), \mu) \leq C \exp\{\beta(m,\|b\|_{\mathcal{C}^\gamma})   (t-t_0)\}.
\end{align}
\end{enumerate}
\end{theorem}
The proof of Theorem \ref{thm inv} is established in Section \ref{sec:inv}. The condition $\beta(2, \| b\|_{\mathcal C^\gamma}) <0$ should be understood as an analogue to the case where $b$ is Lipschitz with a Lipschitz constant $L_b$, where the exponential decay follows if $-\kappa_1 + L_b <0$.

\subsection{Markov property: Proof of Theorem \ref{cor:markov}}\label{subsec:link-SDEs}
Before moving on to proving existence and uniqueness of the invariant measure, the goal of this section is to prove Theorem \ref{cor:markov} which motivates the definition of the invariant measure.
We will always assume that $F$ satisfies Assumptions \ref{assumpF} and \ref{Flip}, $\gamma \in (1-1/(2H),1)$ and $b$ satisfies Assumption \ref{assumpb}. Therefore, existence and uniqueness of solutions to \eqref{eq:sde} (in the class \eqref{eq:class-uniqueness}) and \eqref{eq:sde-eta} (in the class \eqref{eq:class-psi}) will follow from Theorem \eqref{thmmain-uniqueness} and Theorem \eqref{thm:Phi-results}.
We start by proving the joint-measurability of the map $(x,v,\omega) \mapsto X^{(x,v)}(t_0,t)(\omega)$ in Lemma \ref{lem:measurablePsi}. The proof is based on the following proposition, which gives the continuity of $X^{(x,v)}(t_0,t)$ with respect to $v$.

\begin{prop}\label{prop:cont-history}
Let $x \in \R^d$, $v^1,v^2 \in \mathcal{C}([t_0,+\infty), \R^d)$, $T \in \R$ and $t_0 <T$. 
For any $m \in [2,\infty)$, there exists a constant $C:=C(m,\gamma,H,L_F,\Xi,T-t_0)$ such that for all $t \in [t_0,T]$, the following holds
\begin{align}\label{eq:cont-history}
\| X^{(x,v^1)}(t_0,t) -X^{(x,v^2)}(t_0,t)  \|_{L_m} \leq C  \sup_{r \in [t_0,T]} |v^{1}_r-v^{1}_{t_0} - (v^{2}_r-v^2_{t_0}) |   .
\end{align}
\end{prop}
\begin{proof}
Let $x \in \R^d$, $T \in \R$ and $t_0 < T$.
As in the beginning of the proof of Proposition \ref{prop:stab-lip}, we can assume without any loss of generality that $b$ is a smooth function. Moreover, for any continuous function $v \in \mathcal{C}([t_0,+\infty),\R^d)$, define $\tilde{v}=v-v_{t_0}$, then $X^{(x,v)}(t_0,\cdot) = X^{(x,\tilde v)}(t_0,\cdot)$, hence we can also assume without any loss of generality that $v^1_{t_0}=v^2_{t_0}=0$.

For simplicity, we write $X^i_t \equiv X^{(x,v^i)}(t_0,t)$. Let $(\tilde s, \tilde t) \in \Delta_{[t_0,T]}$ such that $\tilde t- \tilde s \leq 1$, and let $(s,t) \in \Delta_{[\tilde s, \tilde t]}$, then
\begin{align*}
X_t^1-X_t^2 - \EE^s(X_t^1-X_t^2) & = \int_s^t F(X_r^1)-F(X_r^2)-\EE^s \left( F(X_r^1)-F(X_r^2) \right) dr \\ & \quad + \int_s^t b(X_r^1) -b(X_r^2) dr -\EE^s \int_s^t b(X_r^1)-b(X_r^2) dr .
\end{align*}
Using the Lipschitz property of $F$ and taking the $L_m$ norm, we get
\begin{align*}
\| X_t^1-X_t^2 - \EE^s(X_t^1-X_t^2)  \|_{L_m} \leq C \int_s^t \|X_r^1-X_r^2\|_{L_m} dr + C\| \int_s^t b(X_r^1) -b(X_r^2) dr \|_{L_m} .
\end{align*}
Using Proposition \ref{prop:reg2} with $\psi=X^1-\widetilde{B}^{t_0}$, $\phi=X^2-\widetilde{B}^{t_0}$, $\tau=1/2$, and recalling that $\llbracket X^1-\widetilde{B}^{t_0} \rrbracket_{\mathcal{C}^{1+H(\gamma \wedge 0)}_{[s,t]} L_{1,\infty}} \leq C$ by \eqref{eq:inf-Phi}, we get that

\begin{align*}
\| X_t^1-X_t^2 - \EE^s(X_t^1-X_t^2)  \|_{L_m} & \leq C \sup_{r \in [S,T]} \|X_r^1-X_r^2\|_{L_m}    (t-s)^{1+H(\gamma-1)} \\ & \quad +C   \llbracket X^1-X^2 \rrbracket_{\mathcal{C}^{1/2}_{[\tilde s, \tilde t]}L_m}  (t-s)^{1+H(\gamma-1)} .
\end{align*}
Dividing by $(t-s)^{1/2}$ and taking the supremum over $(s,t) \in \Delta_{[\tilde s,\tilde t]}$, we get that
\begin{align*}
 \llbracket X^1-X^2 \rrbracket_{\mathcal{C}^{1/2}_{[\tilde s,\tilde t]}L_m}& \leq C   \sup_{r \in [S,T]} \|X_r^1-X_r^2\|_{L_m}  (\tilde t- \tilde s)^{1/2+H(\gamma-1)} \\ & \quad +C    \llbracket X^1-X^2 \rrbracket_{\mathcal{C}^{1/2}_{[\tilde s,\tilde t]}L_m}  (\tilde t-\tilde s)^{1/2+H(\gamma-1)} .
\end{align*}
Let $\ell := (2C)^{\frac{1}{1/2+H(\gamma-1)}} \wedge 1$, then using that $ \llbracket X^1-X^2 \rrbracket_{\mathcal{C}^{1/2}_{[\tilde{s},\tilde{t}]}L_m} <+\infty$ (since $X^1$ and $X^2$ satisfy \eqref{eq:inf-Phi}), we get
\begin{align}\label{eq:holderbound-history}
\llbracket X^1-X^2 \rrbracket_{\mathcal{C}^{1/2}_{[\tilde{s},\tilde{t}]}L_m}& \leq C \sup_{r \in [\tilde{s},\tilde{t}]}\|X_r^1-X_r^2\|_{L_m}  .
\end{align}
Next, we compare $X^1$ and $X^2$ in the supremum norm. Let $t \ge t_0$, then
\begin{align*}
X_t^1-v^{1}_{t}-(X^2_t-v^{2}_{t}) = \int_{t_0}^t F(X_r^1) -F(X_r^2) dr + \int_{t_0}^t b(X_r^1)-b(X_r^2) dr .
\end{align*}
Let $(t_k)_{k \in \llfloor 0,N \rrfloor}$ be a partition of $[t_0,t]$ such that $t_{k+1}-t_k \leq \ell$, then using that $F$ is Lipschitz and taking the $L_m$ norm, we get
\begin{align*}
\| X_t^1-v^{1}_{t}-(X^2_t-v^{2}_{t}) \|_{L_m} \leq \int_{t_0}^t \|X_r^1-X_r^2\|_{L_m} dr + \sum_{k=0}^{N-1} \| \int_{t_k}^{t_{k+1}} b(X_r^1)-b(X_r^2) dr \|_{L_m} .
\end{align*}
Applying Proposition \ref{prop:reg2} as before and using $\llbracket X^1-\widetilde{B}^{t_0} \rrbracket_{\mathcal{C}^{1+H(\gamma \wedge 0)}_{[t_k,t_{k+1}]} L_{1,\infty}} \leq C$, we get that
\begin{align*}
\| X_t^1-\eta^{1}_{t}-(X^2_t-\eta^{2}_{t})  \|_{L_m} & \leq \int_{t_0}^t \|X_r^1-X_r^2\|_{L_m} dr \\ & \quad +C  \sum_{k=0}^{N-1}  \Big(  \sup_{r \in [t_k,t_{k+1}]} \|X_r^1-X_r^2 \|_{L_m} (t_{k+1}-t_k)^{1+H(\gamma-1)} \\ & \quad + \llbracket X^1-X^2 \rrbracket_{\mathcal{C}^{1/2}_{[t_k,t_{k+1}]}L_m}  (t_{k+1}-t_k)^{1+H(\gamma-1)} \Big) .
\end{align*}
Using \eqref{eq:holderbound-history}, we have
\begin{align*}
\| X_t^1-v^{1}_{t}-(X^2_t-V^{2}_{t}) \|_{L_m} & \leq \int_{t_0}^t \|X_r^1-X_r^2\|_{L_m} dr \\ & \quad +C \sum_{k=0}^{N-1}   \sup_{r \in [t_k,t_{k+1}]} \|X_r^1-X_r^2 \|_{L_m} (t_{k+1}-t_k)^{1+H(\gamma-1)} .
\end{align*}
Hence, we conclude that
\begin{align*}
\|  X_t^1-\eta^{1}_{t}-(X^2_t-v^{2}_{t}) \|_{L_m} & \leq C \sum_{k=0}^{N-1}  \Big( \sup_{r \in [t_k,t_{k+1}]} \| X_r^1-\eta^{1}_{r}-(X^2_r-v^{2}_{r}) \|_{L_m}  \\ & \quad +  \sup_{r \in [t_k,t_{k+1}]} \left|  v^{1}_{r}-v^{2}_{r}  \right| \Big)  (t_{k+1}-t_k)^{1+H(\gamma-1)} .
\end{align*}
Finally, a Gr\"onwall-type argument shows that
\begin{align*}
\|  X_t^1-v^{1}_{t}-(X^2_t-v^{2}_{t}) \|_{L_m} & \leq C \sup_{r \in [t_0,T]} \left|  v^{1}_{r}-v^{2}_{r}  \right| ,
\end{align*}
which concludes the proof.
\end{proof}

The following lemma gives the existence of the map $(x,v) \mapsto X(t_0,t,x,v)$ hinted at in \eqref{eq:Xmes}. This is a classic result when $b$ is a smooth function. Indeed, one can fix $t_0 \in \R$, $(x, f) \in \R^d\times \mathcal{C}([t_0, +\infty), \R^d)$ and consider the ODE $du_t=F(u_t)dt + b(u_t)dt + df_t$, $u_{t_0}=x$. Then one can define solution map 
\begin{align}\label{eq:gamma}
\Gamma_b: \R^d \times \mathcal{C}([t_0,+\infty),\R^d) \rightarrow \mathcal{C}([t_0,+\infty),\R^d),
\end{align}
and show that for any $T \ge t_0$, $\Gamma_b$ is continuous from $\R^d \times \mathcal{C}([t_0,+\infty),\R^d)$ to $\mathcal{C}([t_0,T],\R^d)$ (see, e.g., \cite[Proposition 4.3]{da2006introduction}). In this case, for any $t \ge t_0$, one can define $(x,v) \mapsto X(t_0,t,x,v)$ as $$X(t_0,t,x,v) := \Gamma_b(x,v+\widetilde{B}^{t_0})(t),$$
and check that $(X(t_0,t,x,v))_{t \ge t_0}$ solves the SDE \eqref{eq:sde-eta} with initial input $(x,v)$. Hence it follows by uniqueness of solutions in $\mathcal{V}^{v}$ that for any $t \ge t_0$, almost surely we have $X(t_0,t,x,v)=X^{(x,v)}(t_0,t)$. However, such $\Gamma_b$ for $b \in \mathcal{C}^\gamma$ and $\gamma <0$ does not exist and one needs to rely on smooth approximations of $b$ to get the same result. That is the purpose of the following lemma.

\begin{lemma}\label{lem:measurablePsi}
For any $t_0 \in \R$ and $t \ge t_0$, there exists a jointly measurable map  $(x,v,\omega) \in \R^d \times \mathcal{C}([t_0,\infty) ,\R^d) \times \Omega \mapsto X(t_0,t,x,v)(\omega)$ such that for any $(x,v) \in \R^d \times \mathcal{C}([t_0,\infty) ,\R^d)$,
\begin{align*}
    \mathbb{P}(X^{(x,v)}(t_0,t)=X(t_0,t,x,v))=1 .
\end{align*}
\end{lemma}
\begin{proof}
Let $t_0 \in \R$ and $x \in \R^d$.
We take $(b^n)_{n \in \N}$ a sequence of smooth functions converging to $b$ in $\mathcal{C}^{\gamma}$ such that $\sup_{n \in \N} \|b^n\|_{\mathcal{C}^\gamma} \leq \Xi$, we write $X^{n,(x,v)}$ for the solution of \eqref{eq:sde-eta} with drift $b^n$.

Let $t \ge t_0$.
For any $x \in \R^d$ and any continuous function $v \in \mathcal{C}([t_0,+\infty),\R^d)$, define $f(x,v,\omega):= X^{(x,v)}(t_0,t)(\omega)$. Then for each $x,v$, $f(x,v,\cdot)$ is a measurable function, since $X^{(x,v)}$ is measurable in $\omega$ as the limit of $X^{n,(x,v)}$ by \eqref{eq:cont-drift}. Moreover, for any sequence $x_n,v_n$ that converges to $x,v$ in the space $\R^d \times \mathcal{C}([t_0,+\infty),\R^d)$, we have that $X^{(x_n,v_n)}(t_0,t)$ converges to $X^{(x,v)}(t_0,t)$ in probability by Proposition \ref{prop:cont-history} and \ref{eq:decay-Phi} (taking $\kappa_1=-L_F$). So it follows that $f$ is stochastically continuous, and one can apply \cite[Lemma 4.31]{scheutzow2013stochastic} to conclude that $f$ has a jointly measurable modification.
\end{proof}

Let $(X_0,V)$ be an $\mathcal{F}_{t_0}$-measurable random variable. The following lemma confirms an expected result: $X^{(X_0,V)}(t_0,t) = X(t_0,t,X_0,V)$.
This is a classic result when $b$ is a smooth function. Indeed, one can check that the evolution
\begin{align}\label{eq:confusion-smooth}
    \left( X(t_0,t,X_0,V)\right)_{t \ge t_0} =\left( \Gamma_b(X_0,V+\widetilde{B}^{t_0})(t) \right)_{t \ge t_0}
\end{align}
corresponds to the solution of the SDE \eqref{eq:sde-eta} with initial input $(X_0,V)$. Hence, by uniqueness of solutions in the class $\mathcal{V}^V$, it follows that $X(t_0,t,X_0(\omega),V(\omega))(\omega)=X^{(X_0,V)}(t_0,t)(\omega)$. The following lemma extends this result to $b$ in $\mathcal{C}^\gamma$ and is needed in the proof of Theorem \ref{cor:markov}.
\begin{lemma}\label{lem:confusion-xi}
Let $t_0 \in \R$, $(X_0,V): \Omega \mapsto \R^d \times \mathcal{C}([t_0,\infty),\R^d)$ be an $\mathcal{F}_{t_0}$-measurable random variable. Then for any $t \ge t_0$, almost surely in $\omega$, we have
\begin{align}\label{eq:xi=tildexi}
X(t_0,t,X_0(\omega),V(\omega)) (\omega) =    X^{(X_0,V)} (t_0,t)(\omega) .
\end{align}
\end{lemma}
\begin{proof}
Let $t_0 \in \R$ and $x \in \R^d$.
We take $(b^n)_{n \in \N}$ a sequence of smooth functions converging to $b$ in $\mathcal{C}^{\gamma}$ such that $\sup_{n \in \N} \|b^n\|_{\mathcal{C}^\gamma} \leq \Xi$, we write $X^{n,(x,v)}=\Gamma_{b^n}(x,v+\widetilde{B}^{t_0})$ for the solution of \eqref{eq:sde-eta} with smooth drift $b^n$ and initial input $(x,v)$.

Let $t \ge t_0$.
For any $n \in \mathbb{N}$, any $x \in \R^d$ and any $v \in \mathcal{C}([0,+\infty),\R^d)$, define $f^n(x,v,\omega):= |X^{n,(x,v)}(t_0,t)(\omega)-X(t_0,t,x,v)(\omega)| \wedge 1$. Recall that Lemma \ref{lem:measurablePsi} gives the joint measurability of $(x,v,\omega) \mapsto X(x,v,\omega)$. Moreover, since $(x,v,\omega) \mapsto X^{n,(x,\omega)}$ is jointly measurable (this follows from the continuity of $\Gamma_{b^n}$), we also get that $f^n$ is jointly measurable.

Moreover, for any fixed $x,v$, $f^n(x,v,\cdot)$ is independent of $\mathcal{F}_{t_0}$. Then by Lemma \ref{lem:disintegration}, taking $J=(X_0,V)$ we get that
\begin{align*}
    \EE f^n(X_0(\cdot),V(\cdot),\cdot) = \EE g^n(X_0,V),
\end{align*}
where $g^n(x,v)=\EE f^n(x,v,\cdot)$.
Moreover, by \eqref{eq:cont-drift}, we have $$\sup_{(x,v)} g^n(x,v) =\sup_{(x,v)} \EE\|X^{n,(x,v)}(t_0,t)-X^{(x,v)}(t_0,t)\| \leq C \| b^n-b\|_{\mathcal{C}^{\gamma-1}}.$$ Hence, it follows that
\begin{align*}
 \EE |X^{n,(X_0,V)}(t_0,t)- X(t_0,t,X_0,V)| \wedge 1=  \EE f^n(X_0(\cdot),V(\cdot),\cdot)  \leq C \| b^n-b\|_{\mathcal{C}^{\gamma-1}}.
\end{align*}
Recall by \eqref{eq:confusion-smooth} that $(X^{n,(X_0,V)}(t_0,t))_{t \ge t_0}$ corresponds to the solution of \eqref{eq:sde-eta} with smooth drift $b^n$. Hence by Proposition \ref{eq:cont-drift} (stability with respect to the drift), letting $n$ go to infinity, we deduce that
\begin{align*}
\EE |X^{(X_0,V)}(t_0,t)-X(t_0,t,X_0,V)| =0 .
\end{align*}
\end{proof}
\begin{proof}[Proof of Theorem \ref{cor:markov}.]
\textbf{Proof of $(i)$.} 
The existence of the jointly measurable map $(x,v,\omega) \mapsto X(t_0,t,x,v)$ follows from Lemma \ref{lem:measurablePsi} which also states that it is a modification of the map $(x,v) \mapsto X^{(x,v)}(t_0,t)$.

\textbf{Proof of $(ii)$.}
Let $t_0 \in \R$, $(x,v)\in \Rd\times \mathcal{C}([0,\infty),\R^d)$. For any $x \in \R^d$ and $w \in \mathcal{H}_H$, we define $(\Theta(t_0,t,x,w))_{t \ge t_0}$ as in \eqref{def:xi} and the family of operators $(P_{s,t})_{t_0 \leq s \leq t}$ as in \eqref{def:semigroup}.
Let $f$ be a bounded measurable function on $\Rd\times\mathcal{H}_H$ and let $t\ge s \ge t_0$. From \eqref{def:semigroup}, $P_{s,t}f$ is also a bounded measurable function and $P_{s,s}$ is the identity operator. 

Next, we prove that
\begin{align}\label{eq:markovprop0}
\EE^s f(\Theta(t_0,t,x,w)) 
& = P_{s,t} f(\Theta(t_0,s,x,w)) \quad\text{a.s.} .
\end{align}
Note that taking expectation, we obtain that $P_{t_0,t}f = P_{t_0,s}P_{s,t}f$. Since this is true for any $t_0$, it proves that $P$ is an evolution family and therefore that $(\Theta(t_0,t,x,w))_{t\ge t_0}$ is a Markov process with evolution family $P$. In order to prove \eqref{eq:markovprop0}, we prove
equivalently that
\begin{align}\label{eq:markovprop1}
\EE Y f(\Theta(t_0,t,x,w)) 
& = \EE Y P_{s,t} f(\Theta(t_0,s,x,w)) 
\end{align}
for any bounded $\mathcal{F}_s$-measurable random variable $Y$. 
Recall the definition of $Z$ in \eqref{def:Z}. Let us define the evolution $\Theta^{(x,w)}(t_0,\cdot)$ by
$$\Theta^{(x,w)}(t_0,t):= \left( X^{(x,\mathbf{A}^{t_0}(w))}(t_0,t) ,Z(t_0,t,w) \right), \quad t \ge t_0.$$
Using Lemma \ref{lem:confusion-xi}, we can write 
\begin{align*}
\EE Y f(\Theta(t_0,t,x,w)) =\EE Yf(\Theta^{(x,w)} (t_0,t)).
\end{align*}
Since the first component of the process $(\Theta^{(x,w)} (t_0,t))_{t \ge s}$ solves the SDE \eqref{eq:sde-eta} with initial input $\Theta^{(x,w)}(t_0,s)$, it follows by uniqueness of solutions in the class $\mathcal{V}^{Z(t_0,s,w)}$ that
\begin{align*}
\EE Y f(\Theta(t_0,t,x,w)) =\EE Yf(\Theta^{\Theta^{(x,w)}(t_0,s)} (s,t) ).
\end{align*}
Using Lemma \ref{lem:confusion-xi} again, we get
\begin{align*}
 \EE Y f(\Theta(t_0,t,x,w))   = \EE h(Y,\Theta(t_0,s,x,w)),
\end{align*}
where  $\R\times\Rd\times \mathcal{H}_H\times\Omega\ni (y,z,\zeta,\omega)\mapsto h(y,z,\zeta,\omega) := yf(\Theta(s,t,z,\zeta)(\omega))$ is a measurable function (in view of Lemma \ref{lem:measurablePsi} and continuity of the operator $\mathbf{A}^{t_0}$).
Observing that $\EE h(y,z,\zeta)=yP_{s,t}f(z,\zeta)$ and $h$ is independent from $\mathcal{F}_s$, and taking $J:=(Y,\Theta(t_0,s,x,w))$, we can apply Lemma \ref{lem:disintegration} to obtain \eqref{eq:markovprop1}, and deduce that $P$ is an evolution family. Moreover, using the continuity of the operator $\mathbf{A}^{t_0}$, Proposition \ref{prop:cont-history} and the continuity with respect to the initial condition (take $\kappa_1=-L_F$ in \eqref{eq:stab-lip}), one gets that $P_{s,t}$ maps $\mathcal{C}(\R^d \times \mathcal{H}_H, \R)$ to $\mathcal{C}(\R^d \times \mathcal{H}_H, \R)$ for any $t\ge s \ge  t_0$, which proves that $P$ is a Feller evolution family.

\textbf{Proof of $(iii)$.}
Let $x \in \R^d$ and $\mu_0=\delta_x \times \mathbf{W}$ be a generalised initial condition. Let $g: \R^d \times \mathcal{H}_H \rightarrow \R$ be a bounded measurable and Lipschitz continuous function. We show that $(P^*_{s,t}\mu_0) (g)=(P^*_{0,t-s}\mu_0)(g)$ for any $t \ge s$. Note that if this is true for smooth $b$, then using the stability estimate \eqref{eq:cont-drift} and the Lipschitz continuity of $g$, the same equality still holds for $b \in \mathcal{C}^\gamma$. Hence, we assume without any loss of generality that $b$ is a smooth function.

Recall the definition \eqref{def:xi} and \eqref{eq:gamma}.
For any $s \in \R$, define $\bar W^s = (W_{u+s}-W_s)_{u \leq 0}$, and note that for any $s \in \R$, $\bar W^s$ has law $\mathbf{W}$. Let $t \ge s$, using Lemma \ref{lem:confusion-xi}, we have that
\begin{align*}
    P^*_{s,t}\mu_0(g)= \EE g (\Theta(s,t,x,\bar W^s)) & = \EE g(\Theta^{(x,\bar W^s)}(s,t))\\
    & = \EE g \left( X^{(x,\mathbf{A}^{s}(\bar W^s))}(s,t), Z(s,t,\bar W^s) \right).
\end{align*}
Using the stationarity under time-shifts of the fBm and the Wiener process, one can show that $\left(  X^{(x,\mathbf{A}^{s}(\bar W^s))}(s,t) ,Z(s,t,\bar W^s)   \right)$ and $\left(  X^{(x,\mathbf{A}^{0}(\bar W^0))}(0,t-s) ,Z(0,t-s,\bar W^0)   \right)$ have the same law. Hence, it follows that 
\begin{align*}
    P^*_{s,t}\mu_0(g)=  \EE g\left( X^{(x,\mathbf{A}^0(\bar W^0)}(t_0,t), Z(0,t-s,\bar W^0) \right) .
\end{align*}
Using Lemma \ref{lem:confusion-xi} again, we conclude that $P^*_{s,t}\mu_0(g)= \EE g (\Theta(0,t-s,x,\bar W^0)) = P^*_{0,t-s} \mu_0(g)$. Finally, the semigroup property follows from the fact that $P$ is an evolution family.
\end{proof}

\subsection{Existence and uniqueness of an invariant measure: Proof of Theorem \ref{thm inv}}\label{sec:inv}
Recall that we fixed a probability space $(\Omega, \mathcal{F}, \mathbb{F}, \mathbb{P})$ and an $\mathbb{F}$-Wiener process $W$, and that $\mathbf{W}$ is left-sided Wiener measure.
Let $\mu_0 = \delta_0 \times \mathbf{W}$. Define $\bar W= (W_{t})_{t \leq 0}$. Then $(x,\bar{W})$ has law $\mu_0$ and $\mathbf{A}^{0}(\bar{W}) = \bar{B}^{0}$.
\paragraph{Proof of $(i)$.}
We start by proving existence of an invariant measure.
Recall the operator \eqref{def:semigroup} and the definition $(P_t)_{t \ge 0}=(P_{0,t})_{t \ge 0}$. Theorem \ref{cor:markov} shows that $P^*_t\mu_0$ is the transition semigroup for the time-homogeneous Markov process $\left(\Theta(0,t,X_0,\bar W) \right)_{t \ge 0}$.
For any $T>0$, consider the ergodic averages $$\mathcal{R}_T \mu_0 = \frac{1}{T}\int_0^T P^*_{t}\mu_0.$$ 
Let $X_t^{V}$ denote the solution of \eqref{eq:sde-eta} starting from initial input $(0,V)$ at time $0$. Using \eqref{eq:unif-Phi} with $V=\bar B^{0}$ and recalling that $\kappa_1>0$, we have
\begin{align}\label{eq:lyapunov}
\mathbb{E} \left| X^{\bar{B}^{0}}_t\right|^m \leq C, \quad \forall m \ge 2.
\end{align}
Let $R>0$ and let $K_R =(-\infty,-1/\sqrt R) \cup (1/\sqrt R,+\infty)$. Then
\begin{align*}
\mathcal{R}_T \mu_0 (K_R^c)=\frac{1}{T}\int_0^T P_t^*\mu_0(K_R^c) dt & = \frac{1}{T} \int_0^T \int_{\R^d \times \mathcal{H}_H} \mathds{1}_{K_R^c}(x)   P_t^*\mu_0(dx,dw) dt \\ & =  \frac{1}{T} \int_0^T \EE \Big( \EE \mathds{1}_{K_R^c} \left(X(0,t,0,\mathbf{A}^{0}(w)) ) \right) \Big)_{w=\bar W} dt.
\end{align*}
Since $X(0,t,x,v)$ is independent from $\mathcal{F}_{0}$, using Lemma \ref{lem:disintegration} with $J=(x,\bar{B}^{0})$, we get that 
\begin{align*}
\mathcal{R}_T \mu_0 (K_R^c) = \frac{1}{T}\int_0^T \EE \mathds{1}_{K_R^c} (X(0,t,0,\bar{B}^{0})) \leq \frac{1}{R} \EE |X(0,t,0,\bar{B}^{0})|^2 .    
\end{align*}
Using Lemma \ref{lem:confusion-xi} and \eqref{eq:lyapunov} with $m=2$, we conclude that
\begin{align*}
    \mathcal{R}_T \mu_0 (K_R^c) \leq \frac{1}{R} \EE |X^{\bar{B}^{0}}_t|^2 \leq \frac{C}{R} .
\end{align*}

Combining the above with the fact that the projection of $\mathcal{R}_T\mu_0$ over $\mathcal{H}_H$ is $\mathbf{W}$, we get tightness of the sequence $(\mathcal{R}_T \mu_0)_{T > 0}$. Then, recalling that $\R^d \times \mathcal{H}$ is a Polish space and using the Feller property, we obtain, by the standard Krylov-Bogolyubov argument (see, for example, \cite[Section 11.2]{da2014stochastic}), that any weak limiting point is an invariant measure.

Next we show that the invariant measure $\mu$ constructed above has finite moments.
Let $R>0$ and define $\phi: \R^d \times \mathcal{H}_H \rightarrow \R$ by $\phi_R(y,w):=R \wedge |y|^m $. Since $\mu$ is a weak limit of $\mathcal{R}_{T_n} \mu_0$ for some sequence $T_n \rightarrow +\infty$, we have that $$\int_{\R^d \times \mathcal{H}_H} \phi_R(y,w)d \mathcal{R}_{T_n} \mu_0 \rightarrow \int_{\R^d \times \mathcal{H}_H} \phi_R(y,w) d\mu.$$ Moreover, using Lemma \ref{lem:disintegration} and Lemma \ref{lem:confusion-xi}, we have
\begin{align*}
\int_{\R^d \times \mathcal{H}_H} \phi_R(y,w)d \mathcal{R}_{T_n} \mu_0 &= \frac{1}{T} \int_0^T \EE \left( \EE \phi_R(\Theta(0,t,0,w)) \right)_{w=\bar W}
\\ & = \frac{1}{T} \int_0^T \EE \phi_R(\Theta^{(0,\bar W)} (0,t)) dt .    
\end{align*}
The uniform-in-time bound on the moments \eqref{eq:lyapunov} implies that $$\int_{\R^d \times \mathcal{H}_H} \phi_R(y,w)d \mathcal{R}_{T_n} \mu_0 \leq C.$$ Passing to the limit in $n$, we get
$$\int_{\R^d \times \mathcal{H}_H} \phi_R(y) d\mu  \leq C.$$ Finally, by the monotone convergence theorem, we deduce that
\begin{align}\label{eq:inv-moments}
    \int_{\R^d \times \mathcal{H}_H} |y|^m d\mu \leq C .
\end{align}

\paragraph{Proof of $(ii)$.}
We prove uniqueness via an exponential contraction under the assumption
that 
\begin{align*}
-\kappa_1+\mathbf{M}(2,\gamma,H,L_F,\Xi) \|b \|_{\mathcal{C}^{\gamma}} < 0 ,
\end{align*}
where recall that $\kappa_1$ is the dissipation constant in \eqref{eq:Fdissipative} and $\mathbf{M}(2,\gamma,H,L_F)$ is the constant from \eqref{eq:decay-Phi}. Since $\R^d \times \mathcal{H}_H$ is a Polish space and the marginals of $\mu$ and $\nu$ on $\mathcal{H}_H$ are the same, using the gluing lemma (\cite{villani2008optimal}), one can construct a triple $(X_0,\bar W,Y_0)$ on a probability space $(\Omega',\mathcal{F}',\mathbb{P}')$ such that $(X_0,\bar W)$ has law $\mu$ and $(Y_0,\bar W)$ has law $\nu$.

Let $t_0 \in \R$. On the extended probability space $$(\widetilde \Omega:= \Omega' \times  \Omega, \widetilde{\mathcal{F}} := \mathcal{F}' \times \mathcal{F},  \widetilde{\mathbb P}:= \mathbb{P}'\times \mathbb{P}),$$ we consider for any $t \ge t_0$ the random variables
$$ \Theta(t_0,t,X_0, \bar W )  \ \text{and} \  \Theta(t_0,t,Y_0, \bar W ) .$$ 
By definition of the invariant measure, for any bounded measurable function $f: \R^d \times \mathcal{H}_H \rightarrow \R$, we have for any $t \ge t_0$
\begin{align*}
 \widetilde \EE f(X_0,\bar W) & = \EE_{\mathbb P'} \left( \EE_{\mathbb P} f(\Theta(t_0,t,x,w)) \right)_{(x,w)=(X_0,\bar W)} \\
  \widetilde \EE f(Y_0,\bar W) & = \EE_{\mathbb P'} \left( \EE_{\mathbb P} f(\Theta(t_0,t,y,w)) \right)_{(y,w)=(Y_0,\bar W)} .
\end{align*}
Hence, using Lemma \ref{lem:disintegration}, we deduce that $\Theta(t_0,t,X_0, \bar W )$ and $\Theta(t_0,t,Y_0, \bar W )$ have respective laws $\mu$ and $\nu$ for any $t$. Therefore, equipping the space $\R^d \times \mathcal{H}_H$ with the norm
\begin{align*}
    \| (x,w) \| = |x| + \|w\|_{\mathcal{H}_H},
\end{align*}
we have that for any $m \in [2,\infty)$,
\begin{align*}
    \mathcal{W}_m(\mu,\nu) & \leq \big( \widetilde \EE \|\Theta(t_0,t,X_0, \bar W ) -\Theta(t_0,t,Y_0, \bar W ) \|^m \big)^{1/m} \\ & = \big(  \widetilde \EE|X(t_0,t,x, \mathbf{A}^{t_0}(\bar W) ) -X(t_0,t,y, \mathbf{A}^{t_0}(\bar W) ) |^m  \big)^{1/m} .
\end{align*}
Using Lemma \ref{lem:confusion-xi}, \eqref{eq:decay-Phi} and \eqref{eq:inv-moments}, it follows that
\begin{align*}
 \mathcal{W}_m(\mu,\nu) & \leq C \exp\{-\kappa_1 (t-t_0)+ \mathbf{M}(m,\gamma,H,L_F,\Xi)\|b\|_{\mathcal{C}^\gamma}   (t-t_0)\} (\EE_{\mathbb{P'}} |X_0-Y_0|^m)^{1/m}.
\end{align*}
Taking $m=2$ and letting $t$ go to $+\infty$, we conclude that $\mathcal{W}_2(\mu,\nu) = 0$, so $\mu=\nu$. Choosing $(X_0,\bar W)$ to have law $\mu_0$ rather than $\mu$, one can run the same argument to show that 
\begin{align*}
\mathcal{W}_m(\mathcal{L}(\Theta(t_0,t,X_0,\mathbf{A}^{t_0}(\bar W)), \mu) \leq C \exp\{\left( -\kappa_1 +  \mathbf{M}(m,\gamma,H,L_F,\Xi)\|b\|_{\mathcal{C}^\gamma}  \right)   (t-t_0)\},
\end{align*}
and use Lemma \ref{lem:confusion-xi} to get \eqref{eq:conv-inv}.

\appendix

\section{Auxiliary results}\label{app:A}
Let $t_0 \in \R$ and recall the decomposition of $B_{t}-B_{t_0}$ into the history $\bar{B}^{t_0}_{t}$ and the innovation $\widetilde{B}^{t_0}_{t}$ given in \eqref{eq:decomp-fBm} for any $t \ge t_0$. For simplicity, we write $\widetilde{B} := \widetilde{B}^{t_0}$.
The following regularity estimates on the conditional expectation of the fractional Brownian motion are adapted from \cite[Propositions 3.6, 3.7 and 3.8]{butkovsky2021approximation}, where they are written for $B$ instead of $\widetilde{B}$ and are based on Gaussian properties of $B$ and the local nondeterminism property
\begin{align*}
\EE \left(  B_t -\EE^s B_t \right)^2 = C (t-s)^{2H},
\end{align*}
all of which also hold for $\widetilde{B}$.

\begin{lemma}\label{lemreg-B}
Let $(\Omega,\mathcal{F},\mathbb{F},\mathbb{P})$ be a filtered probability space and $W$ be an $\mathbb{F}$-Wiener process. Let $B$ be an fBm defined via $W$ by \eqref{eq:MVN}. 
Let $\beta \in (-\infty,1)$. Then there exists a constant $C>0$ such that for any bounded measurable function  $f \in \mathcal{C}^1$ and any $\mathcal{F}_{s}$-measurable random variables $Z_{1}, Z_2,Z_3,Z_4$, and any $t_0 < s < t$, the following holds
\begin{enumerate}[label=(\roman*)]
\item $| \mathbb{E}^s f( \widetilde{B}_t + Z_1 ) | \leq C \|f\|_{\mathcal{C}^\beta} (t-s)^{H(\beta \wedge 0)}$\text{;}

\item $| \mathbb{E}^s \big( f( \widetilde{B}_t + Z_1 ) -f( \widetilde{B}_t + Z_2 )  \big)  | \leq C \|f\|_{\mathcal{C}^{\beta}} |Z_1-Z_2| (t-s)^{H(\beta-1)}$\text{;} 

\item and
\begin{align*}
& | \mathbb{E}^s \big( f( \widetilde{B}_t + Z_1 ) -f( \widetilde{B}_t + Z_2 ) -f( \widetilde{B}_t + Z_3 ) +f( \widetilde{B}_t + Z_4 )   \big)  |\\ &  \leq C \|f\|_{\mathcal{C}^{\beta}} |Z_1-Z_2-Z_3+Z_4| (t-s)^{H(\beta-1)}  \\ & \quad + C \|f\|_{\mathcal{C}^{\beta}} |Z_1-Z_2| | Z_1-Z_3|  (t-s)^{H(\beta-2)} \text{.} 
\end{align*}
\end{enumerate}
\end{lemma}

We recall from \cite{dareiotis2024regularisation}  an extension of the stochastic sewing lemma of \cite{le2020stochastic}, which is a key tool in the regularisation by noise analysis of  Section~\ref{sec:well-posedness}.

\begin{lemma}[Stochastic sewing lemma]\label{lemSSL}
Let $t_0 \leq S<T$, $m \in [2, \infty)$ and $q \in [m,\infty]$. Let $(\Omega,\mathcal{F},\mathbb{F},\mathbb{P})$ be a filtered probability space. Let $A : \Delta_{[S,T]} \rightarrow L_m$ be such that $A_{s,t}$ is $\mathcal{F}_t$-measurable for any $(s,t) \in \Delta_{[S,T]}$. Assume that there exist constants $\Gamma_1,\Gamma_2\geq 0$ and $\varepsilon_1,\varepsilon_2>0$ such that for any $(s,t) \in \Delta_{[S,T]}$ and $u = (s+t)/2$,
\begin{align}
    \|\EE^s[\delta A_{s,u,t}]\|_{L_q}&\leq \Gamma_1 \,  (t-s)^{1+\varepsilon_1},\label{sts1}\\
   \|  \delta A_{s,u,t}  \|_{L_{m,q}^{\mathcal{F}_S}} &\leq \Gamma_2 \, (t-s)^{\frac{1}{2}+\varepsilon_2}. \label{sts2}
\end{align}
Then there exists a unique process $(\mathcal{A}_t)_{t\in [S,T]}$ such that $\mathcal{A}_S=0$ and for all $(s,t) \in \Delta_{[S,T]}$, the following bounds hold for some constants $K_1,K_2>0$
\begin{align} \label{sts3}
   \|   \mathcal{A}_t-\mathcal{A}_s - A_{s,t}  \|_{L_{m,q}^{\mathcal{F}_S}} \leq K_1 (t-s)^{1/2+\varepsilon_2} \quad \text{and} \quad \| \EE^s (\mathcal{A}_t-\mathcal{A}_s-A_{s,t} ) \|_{L_q} \leq K_2 (t-s)^{1+\varepsilon_1} .
\end{align} 

Moreover, there exists a constant $C=C(\varepsilon_1,\varepsilon_2,m, \alpha_1, \alpha_2)$ independent of $S,T$ such that for every $(s,t) \in \Delta_{[S,T]}$ we have
\begin{align*}
    \| \mathcal{A}_t-\mathcal{A}_s \|_{L_{m,q}^{\mathcal{F}_S}} \leq C\, \Gamma_1  (t-s)^{ 1+\varepsilon_1} + C \, \Gamma_2   (t-s)^{ \frac{1}{2}+\varepsilon_2} .
\end{align*}
\end{lemma}

Finally, we recall a result regarding conditional expectations that is useful in the proof of Theorem \ref{cor:markov}.

\begin{lemma}\label{lem:disintegration}
Let $(\Omega,\mathcal F,\mathbb{P})$ be a probability space and let $E$ be a topological space
space endowed with its Borel $\sigma$-algebra $\mathcal B(E)$.  Let $J:\Omega\to E$ be an $E$-valued random variable, and $\mathcal F_J:=\sigma(J)$. Let $\mathcal{G} \subseteq \mathcal{F}$ be a sigma algebra independent of $\sigma(J)$. Let $f:E\times\Omega\to\R$ be a bounded $(\mathcal B(E)\otimes\mathcal G)$--measurable function, then
\begin{equation}\label{eq:key}
  \EE f\big(J(\cdot),\cdot\big) = \EE g_f\big(J(\cdot)\big),
\end{equation}
where $g_f(x):=\EE f(x,\cdot)$.
\end{lemma}

\begin{proof}
Let $A\in\mathcal B(E)$ and $B\in\mathcal G$ and set
$U:=A\times B$.  Assume first that $f(x,\omega) :=\mathds{1}_A(x)\,\mathds{1}_B(\omega)$. By the tower property of conditional expectation and the independence between $\mathds{1}_B$ and $J$, we obtain
\begin{align*}
  \EE f(Y(\cdot),\cdot) = \EE \mathds{1}_A\bigl(J\bigr)\,\mathds{1}_B
  = \EE\,\EE\bigl[ \mathds{1}_A\bigl(J\bigr)\,\mathds{1}_B \mid\mathcal F_J \bigr]
  = \EE \bigl[\mathds{1}_A(J)\,\EE[\mathds{1}_B]\bigr]
  = \EE g_f(J(\cdot),\cdot).
\end{align*}
Hence \eqref{eq:key} holds for all $f$ of the form $\mathds{1}_A(x)\,\mathds{1}_B(\omega)$. Linearity of both sides of \eqref{eq:key} immediately extends the identity to finite linear combinations of indicators of rectangles, i.e. to all $ f(x,\omega) =\sum_{k=1}^{m}\alpha_k\,\mathds{1}_{A_k}(x)\,\mathds{1}_{B_k}(\omega), \alpha_k\in\R$. Let $$\mathscr A:=\Bigl\{\bigcup_{k=1}^m \bigl(A_k\times B_k\bigr): A_k\in\mathcal B(E),\;B_k\in\mathcal G,\;m\in\N\Bigr\}$$
be the algebra generated by measurable rectangles.  Define the vector space
$$\mathcal H:=\Bigl\{k:\,\text{bounded, }(\mathcal B(E)\otimes\mathcal G) \text{-measurable and }\EE k(J(\cdot),\cdot)=\EE g_k(J(\cdot),\cdot)\}$$
where $g_k$ is obtained from $k$ exactly as $g_f$ was obtained from $f$.
The previous arguments show that $\mathds{1}_U\in\mathcal H$ for every $U\in\mathscr A$.
Moreover, by linearity and the bounded monotone convergence theorem, $\mathcal
H$ is closed under bounded monotone point-wise
limits.

Hence, by the functional monotone-class theorem (e.g. \cite[Theorem 5.2.2]{durrett2019probability}), $\mathcal{H}$ contains all bounded functions that are $\sigma(\mathscr A)$-measurable. Since $\sigma(\mathscr A)=\mathcal B(E)\otimes\mathcal G$, we conclude that every bounded $\mathcal{B}(E) \times \mathcal{G}$-measurable $k$ lies in $\mathcal H$. This completes the proof.
\end{proof}

\newcommand{\etalchar}[1]{$^{#1}$}

\end{document}